\documentclass[12pt]{amsart}
\usepackage[utf8]{inputenc}

\usepackage[english]{babel}
\usepackage[T1]{fontenc}

\usepackage{colortbl}
\usepackage[usenames,dvipsnames]{xcolor}

\usepackage{amsmath}
\usepackage{enumitem}
\usepackage{amssymb,amsfonts,stmaryrd,mathrsfs}
\usepackage{amsmath,latexsym,amsthm,mathtools,bbm}
\usepackage{pgf,tikz}

\usepackage{graphicx,xspace}
\usepackage{subcaption}

\definecolor{green}{RGB}{43,92,47}
\definecolor{red}{RGB}{191,0,0}
\usepackage[colorlinks,cite color=red,link color=green,pagebackref=true]{hyperref}
\usepackage{todonotes}

\usepackage[capitalize]{cleveref}

\newtheorem{thm}{Theorem}[section]
\newtheorem{cor}[thm]{Corollary}
\newtheorem{lem}[thm]{Lemma}
\newtheorem{prop}[thm]{Proposition}
\newtheorem{defi}[thm]{Definition}
\newtheorem{conj}{Conjecture}

\theoremstyle{remark}
\newtheorem{rmq}{Remark}

\newcommand{\bfp}{\mathbf{p}}
\newcommand{\bfq}{\mathbf{q}}
\newcommand{\bfr}{\mathbf{r}}
\newcommand{\bfx}{\mathbf{x}}

\newcommand{\B}{\mathcal B^{(\alpha)}}
\newcommand{\C}{\mathcal C^{(\alpha)}}
\newcommand{\Cone}{\mathcal C^{(1)}}
\newcommand{\Bone}{\mathcal B^{(1)}}
\newcommand{\G}{\mathcal G^{(\alpha)}}
\newcommand{\tG}{\widetilde{\mathcal G}^{(\alpha)}}
\newcommand{\Gone}{\mathcal G^{(1)}}
\newcommand{\Binf}{\mathcal B_\infty^{(\alpha)}}

\newcommand{\Span}{\mathrm{Span}}
\newcommand{\mcP}{\mathcal{P}}
\newcommand{\mcA}{\mathcal{A}}
\newcommand{\PY}{\mathcal{P}_{Y}}
\newcommand{\tildePY}{\mathcal{\widetilde P}_{Y}}

\newcommand{\mcS}{\mathcal{S}_{\alpha}}
\newcommand{\mcSstar}{\mathcal{S}_{\alpha}^*}

\newcommand{\J}{J^{(\alpha)}}

\newcommand{\Jch}{\theta^{(\alpha)}}

\newcommand{\Vcirc}{\mathcal{V}_\circ}

\newcommand{\tf}{\nu_\diamond}

\newcommand{\GY}{\Gamma_Y}

\newcommand{\hypermap}{\mathcal{H}}
\newcommand{\tplus}{\lambda^+}
\newcommand{\tminus}{\lambda^-}
\newcommand{\tblack}{\lambda^\bullet}

\newcommand{\ohypermap}{\mathcal{OH}}
\newcommand{\oprehypermap}{\mathcal{OPH}}
\newcommand{\BFC}{{\normalfont BFC} }

\DeclareMathOperator{\weight}{weight}
\DeclareMathOperator{\hG}{\mathit{\widehat G}}
\DeclareMathOperator{\htau}{\widehat \tau}

\title[Differential equations for the series of hypermaps]{Differential equations for the series of hypermaps with control on their full degree profile}
\author{Houcine Ben Dali}
\address{\parbox{\linewidth}{Université de Lorraine, CNRS, IECL, F-54000 Nancy\\
  Universit\'e de Paris, CNRS, IRIF, F-75006 Paris, France\\
  Department of Mathematics, Harvard University, Cambridge, MA 02138, U.S.A.}
}
\email{bendali@math.harvard.edu}

\begin{document}

\begin{abstract}
    We consider the generating series of oriented and non-oriented hypermaps with controlled degrees of vertices, hyperedges and faces. It is well known that these series have natural expansions in terms of Schur and Zonal symmetric functions, and with some particular specializations, they satisfy the celebrated KP and BKP equations.
    
    We prove that the full generating series of hypermaps satisfy a family of differential equations. 
    We give a first proof which works for an $\alpha$ deformation of these series related to Jack polynomials. This proof is based on a recent construction formula for Jack characters using differential operators. We also provide a combinatorial proof for the orientable case.

    Our approach also applies to the series of $k$-constellations with control of the degrees of vertices of all colors. In other words, we obtain an equation for the
generating function of Hurwitz numbers (and their $\alpha$-deformations) with
control of full ramification profiles above an arbitrary number of
points. Such equations are new even in the orientable case.
\end{abstract}
\maketitle

\section{Introduction}

\subsection{Maps}
 A \textit{connected map} is a cellular embedding of a  connected graph into a closed surface without boundary, orientable or not. In this paper, a \textit{map} is an unordered collection of connected maps. We say that a map is  \textit{orientable} if each one of its connected components is embedded on a orientable surface. We will use the word \textit{non-oriented} for maps on general surfaces, orientable or not. Maps appear in various branches of algebraic combinatorics, probability and physics. The study of maps involves various methods such as generating series, matrix integral techniques and bijective methods, see e.g  \cite{LandoZvonkin2004,Eynard2016,BenderCanfield1986,Chapuy2011,AlbenqueLepoutre2020,ChapuyDolega2022}.

A \textit{hypermap} is a map whose faces are colored in two colors $(+)$ and $(-)$, and such that each edge is incident to two faces of different colors. Usually the faces of one color are called \textit{hyperedges}, and the faces of the other color are the \textit{faces} of the hypermap. Hypermaps were first introduced by Cori in \cite{Cori1975} and are in bijection with bipartite maps by duality \cite{walsh1975}.

We consider generating series of hypermaps with three alphabets of variables controlling the degrees of vertices, faces of color $(+)$ and faces of color $(-)$, see \cref{eq gen series oriented,eq gen series non oriented} below. Representation theory allows one to relate these generating series, in the orientable and the non-orientable cases  to Schur and Zonal symmetric functions, respectively \cite{JacksonVisentin1990,GouldenJackson1996}.

The generating series in which we keep one alphabet $\bfp$ and we replace the alphabets $\bfq$ and $\bfr$ by two variables $u$ and $v$ are well studied and are known to be functions of the KP hierarchy (resp. BKP hierarchy) in the orientable (resp. the non-orientable) case; see \cite{Kharchev91,VandeLeur2001}. 

When we keep two alphabets and we only specialize the third one (see \cref{ssec Jack characters}), the hypermap series also satisfies differential equations related to the integrable 2-Toda hierarchy, see \textit{e.g} \cite{AdlervanMoerbeke2001,BousquetSchaeffer2002,EynardOrantin2007}.
Moreover, it has been recently proved in \cite{ChapuyDolega2022} that the two-alphabet series satisfies a family of decomposition equations, which are a sort of Tutte equations.

However, studying the full generating series of hypermaps \textit{i.e.} without any specialization of the three alphabets of variables, is known to be a hard problem: the usual decomposition a la Tutte does not seem to exist and we are not aware of any differential equations satisfied by these series.

The main contribution of this paper is to prove that the generating series of orientable and non-orientable hypermaps satisfy a family of differential equations; see \cref{thm diff eq} below. We also prove that these equations characterize the generating series. These results are established for a more general series which depends on a deformation parameter $\alpha$, and which gives the series of orientable hypermaps when $\alpha=1$, and the series of the non-orientable hypermaps when $\alpha=2$. The result is based on a recent result of Maciej Do\l{}e\k{}ga and the author about Jack characters \cite{BenDaliDolega2023}, see also \cref{thm Jch}.

\subsection{Generating series of hypermaps}
Throughout the paper, we use straight letters to denote series ($H$,$G$,\dots), and curved letters to denote operators ($\B_n$, $\C_\ell$, $\G$,\dots) or linear spaces ($\mcA$, $\mcS$,\dots).

We start by some definitions related to hypermaps.
\begin{defi}\label{def profile hypermap}
The \textit{size} of a map $M$ is its number of edges, and will be denoted $|M|$.
If $M$ is a hypermap, then we associate to it three integer partitions of size $|M|$: 
\begin{itemize}
    \item its \textit{vertex type}, denoted $\lambda^\bullet(M)$, is the partition obtained by reordering the vertex degrees divided\footnote{When we turn around a vertex in a hypermap, colors $(+)$ and $(-)$ alternate. By consequence, each vertex has necessarily even degree.} by 2. 
    \item its \textit{$(+)$ type}, denoted $\tplus(M)$, is the partition obtained by reordering the degrees of the $(+)$ faces.
    \item its \textit{$(-)$ type}, denoted $\tminus(M)$, is the partition obtained by reordering the degrees of the $(-)$ faces.
\end{itemize} 
The \textit{profile} of $M$, is then the tuple of partitions $(\lambda^\bullet(M),\tplus(M),\tminus(M))$. Finally, we say that a hypermap (oriented or not) is \textit{vertex labelled} if:
\begin{enumerate}
    \item for each $d\geq 1$, vertices of same degree $2d$ are labelled $v_{d,1},v_{d,2},\dots$. 
    \item each vertex has a marked oriented corner in a face colored $(+)$. This corner is called the \textit{vertex root}.
\end{enumerate}

\end{defi}

When a map is orientable, we can choose for each one of the connected surfaces an orientation, which will be called the \textit{direct orientation of the surface}. Once this orientation is fixed, we say that the map is oriented. An oriented hypermap is said vertex labelled if vertices are numbered as in item \textit{(1)} above, and if
\textit{
\begin{enumerate}[label={(\arabic*$'$})]
\setcounter{enumi}{1}
\item each vertex has a marked $(+)$ corner,  oriented in the direct orientation, called the \textit{vertex root}.
\end{enumerate}
}

We give in \cref{fig oriented hyper map} an example of an oriented vertex labelled hypermap. Notice that if we start from a hypermap whose faces of color $(-)$ have all degree 2, then we can glue the double edges forming $(-)$ faces in order to obtain a map with only $(+)$ faces. Hence, simple maps (maps with uncolored faces) can be seen as hypermaps for which $\lambda^{-}=[2,2,\dots,2].$

\begin{figure}[t]
    \centering
    \includegraphics[width=0.5\textwidth]{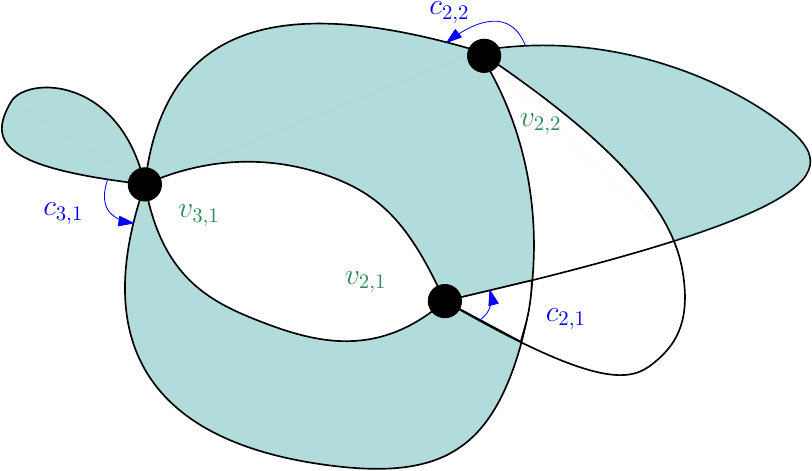}
    \caption{An example of a vertex labelled hypermap of profile $([3,2,2],[5,2],[6,1])$. Faces of color $(-)$ are represented in blue, the root of the vertex $v_{d,i}$ is denoted by $\color{blue}c_{d,i}$.}
    \label{fig oriented hyper map}
\end{figure}

Understanding hypermaps with controlled profile is a hard combinatorial problem. Indeed, usual techniques to enumerate maps do not allow to control the three partitions of the profile and the only known answer to this question is given by an algebraic approach, which provides an expression of the generating series of hypermaps using symmetric functions.

In order to define these generating series, we consider a variable $t$ and three alphabets of variables $\bfp=(p_i)_{i\geq 1}$, $\bfq=(q_i)_{i\geq 1}$ and $\bfr=(r_i)_{i\geq 1}$. For any integer partition $\lambda:=[\lambda_1,\lambda_2,\dots,\lambda_s]$, we define $p_\lambda$ as follows;
$$p_\lambda:=p_{\lambda_1}p_{\lambda_2}\dots p_{\lambda_s}.$$
We define $q_\lambda$ and $r_\lambda$ in a similar way. These three alphabets will be the respective weights of vertices, $(+)$ and $(-)$ faces. More precisely, we consider the two generating series
\begin{equation}\label{eq gen series oriented}
  H^{(1)}(t,\bfp,\bfq,\bfr)=\sum_{M}\frac{t^{|M|}}{z_{\tblack(M)}}p_{\lambda^\bullet(M)}q_{\tplus(M)}r_{\tminus(M)},  
\end{equation}
\begin{equation}\label{eq gen series non oriented}
H^{(2)}(t,\bfp,\bfq,\bfr)=\sum_{M}\frac{t^{|M|}}{2^{ \ell(\tblack(M))}z_{\tblack(M)}}p_{\lambda^\bullet(M)}q_{\tplus(M)}r_{\tminus(M)},    
\end{equation}
where the sum runs over vertex labelled oriented (resp. non-oriented) hypermaps in \cref{eq gen series oriented} (resp. \cref{eq gen series non oriented}). Here, $z_\lambda$ is a normalization numerical factor  (see \cref{subsec Partitions}) related to the labelling of vertices.
By the remark above, the generating series of simple maps are obtained from $H^{(1)}$ and $H^{(2)}$ by taking the specialization $r_2=1$ and $r_i=0$ for $i\neq 2$. It is also worth mentioning that the series $H^{(1)}$ and $H^{(2)}$ are symmetric in the three alphabets $\bfp,\bfq$ and $\bfr$: the symmetry between $\bfq$ and $\bfr$ is clear from the definition and symmetry between $\bfp$ and $\bfq$ can be seen using a duality operation which exchanges vertices and $(+)$ faces (see \textit{e.g} \cite[Definition 2.4]{ChapuyDolega2022}). 

Representation theory tools can be used to give an explicit  expression of the series $H^{(1)}$ and $H^{(2)}$ in terms of symmetric functions (Schur and Zonal functions); see \cref{thm coef c b=0}. These results use an encoding of hypermaps with permutations and matchings. The disadvantage of this approach is that it is quite rigid and it is hard to generalize to the case of generating series of maps with additional weights; an example of such a problem is given by the Matching-Jack conjecture of Goulden--Jackson \cite{GouldenJackson1996}  (see also \cref{conj MJ} below).

The main contribution of this paper is to establish a differential equation for the generating series of hypermaps (\cref{thm diff eq} below) and to solve this equation by giving a recursive formula for the number of hypermaps with fixed profile $(\pi,\mu,\nu)$, see \cref{eq recursive formula}. This recursive formula is given for a one parameter deformation related to Jack polynomials and also to the enumeration of maps with non-orientability weights.

In order to understand the recursive structure of a family of coefficients indexed by three partitions of the same size, it is sometimes  convenient to start by introducing a generalized family of coefficients indexed by partitions of arbitrary size (see \cite{alexanderssonferay2017} for an example of application of this idea). It turns out that one way to make such a generalization in the case of hypermaps is by marking faces of degree 1.
This operation will be justified algebraically in \cref{ssec structure coeff} and a natural combinatorial interpretation in terms of "partially constructed" hypermaps is given for orientable maps in \cref{sec alpha 1}.

\begin{defi}\label{def hypermaps}
Let $\pi$, $\mu$ and $\nu$ be three partitions. We denote by $\ohypermap^\pi_{\mu,\nu}$ (resp. $\hypermap^\pi_{\mu,\nu}$)  the set of vertex labelled oriented hypermaps (resp. non-oriented hypermaps) of profile $(\pi,\mu\cup1^{|\pi|-|\mu|},\nu\cup1^{|\pi|-|\nu|})$, with $|\pi|-|\mu|$ marked $(+)$ faces of degree 1, $|\pi|-|\nu|$ marked $(-)$ faces of degree 1, with the condition that in an isolated loop\footnote{An isolated loop is a connected component of a hypermap with exactly one edge and one vertex. It has necessarily two faces, one colored $(+)$ and one colored $(-)$.}, we can not have both the $(+)$ face and the $(-)$ face marked. By definition, $\ohypermap^\pi_{\mu,\nu}$ and $\hypermap^\pi_{\mu,\nu}$ are empty  when $|\pi|<\max(|\mu|,|\nu|)$.
\end{defi}

We then define the generating series of hypermaps with marked faces;
\begin{equation*}
  \widetilde H^{(1)}(t,\bfp,\bfq,\bfr)=\sum_{\pi,\mu,\nu}\frac{|\ohypermap^\pi_{\mu,\nu}|}{z_\pi}t^{|\mu|+|\nu|-|\pi|} p_\pi q_\mu r_\nu,  
\end{equation*}
\begin{equation*}
  \widetilde H^{(2)}(t,\bfp,\bfq,\bfr)=\sum_{\pi,\mu,\nu}\frac{|\hypermap^\pi_{\mu,\nu}|}{2^{ \ell(\pi)}z_\pi}t^{|\mu|+|\nu|-|\pi|}p_\pi q_\mu r_\nu,  
\end{equation*}
The integer $|\mu|+|\nu|-|\pi|$ corresponds to the number of edges of the hypermap which are not incident to a marked face.
It is straightforward from the definitions that the series $H^{(1)}$ and $H^{(2)}$ are respectively the homogeneous parts of $\widetilde H^{(1)}$ and $\widetilde H^{(2)}$. Conversely, $\widetilde H^{(1)}$ and $\widetilde H^{(2)}$ can be obtained from $H^{(1)}$ and $H^{(2)}$ by simple operations (see \cref{eq c-g 2}).

\subsection{Jack characters}\label{ssec Jack characters}
In order to define a deformed series which interpolates between the series $\widetilde H^{(1)}$ and $\widetilde H^{(2)}$, we use \textit{Jack polynomials}.
Jack polynomials $\J_\lambda$ are symmetric functions introduced by Jack \cite{Jack1970}. They are indexed by an integer partition $\lambda$ and depend on a deformation parameter $\alpha$ (See \cref{thm Jack polynomials} for a precise definition). They can be seen as generalization of Schur functions which are obtained by setting $\alpha=1$.  Jack polynomials have a rich combinatorial structure \cite{Stanley1989,KnopSahi1997,ChapuyDolega2022,Moll2023} and are related to various models of statistical and quantum mechanics \cite{LapointeVinet1995,DumitriuEdelman2002,For}. 

Jack characters are an $\alpha$-deformation of the characters of symmetric group introduced by Lassalle in \cite{Lassalle2008b}.
They have been a useful tool to understand asymptotic behavior of large Young diagrams under a Jack deformation of the Plancherel measure \cite{DolegaFeray2016,Sniady2019,CuencaDolegaMoll2023}.
The Jack character $\Jch_\mu$ is the function on Young diagrams defined by:
$$ \theta^{(\alpha)}_\mu(\lambda):=\left\{   \begin{array}{ll}0 & \mbox{ if } |\lambda|<|\mu|,\\\binom{|\lambda|-|\mu|+m_1(\mu)}{m_1(\mu)}[p_{\mu,1^{|\lambda|-|\mu|}}]J^{(\alpha)}_\lambda &\mbox{ if }|\lambda|\geq|\mu|,   \end{array}\right.$$
where $m_1(\mu)$ is the number of parts equal to 1 in the partition $\mu$, and $p_\mu$ is the power-sum symmetric function indexed by $\mu$ (see \cref{ssec symmetric functions}).

\medskip In \cite{ChapuyDolega2022}, Chapuy and Do\l{}e\k{}ga have introduced  a new family of differential operators $\B_n$ which can be used to construct an $\alpha$ deformation of the generating series of hypermaps\footnote{Actually, they use bipartite maps which can be obtained from hypermaps by duality (a bijection which exchanges the vertices of the map with its faces).} with controlled vertex degrees, $(+)$ face degrees, and the \textit{number} of $(-)$ faces. This corresponds to generating series in which we keep two alphabets $\bfp$ and $\bfq$ and replace each $r_i$ by the same variable $u$ for $i\geq 1$. The definition of operators $\B_n(\bfp,u)$ uses catalytic variables and will be recalled in \cref{ssec op Bn}. These operators have also a combinatorial interpretation for general $\alpha$ which we recall in \cref{ssec C alpha 1} for $\alpha=1$. As an example, we give here the first two operators;
$$\B_1(\bfp,u)=\frac{u p_1}{\alpha}+\sum_{i\geq 1}p_{i+1}\frac{i \partial}{\partial p_i},$$
\begin{multline*}\B_2(\bfp,u)=\frac{u^2 p_2}{\alpha}+\sum_{i\geq 1}\left(\big(2 u+(i+1)(\alpha-1)\big) p_{i+2} +\sum_{j+k=i+2\atop{j,k\geq 1}}p_{j}p_k\right)\frac{i \partial}{\partial p_i}\\
  +\frac{u}{\alpha}\big( (\alpha-1) p_2+p_{1,1}\big)+\alpha\sum_{i,j\geq 1}p_{i+j+2}\frac{i\partial}{\partial p_i} \frac{j\partial}{\partial p_j}.  
\end{multline*}

It turns out that Jack characters can be constructed using these differential operators and, consequently, they have a combinatorial interpretation in terms of a family of weighted maps
(see \cite[Theorem 1.4 and Proposition 4.6]{BenDaliDolega2023}).
\begin{thm}[\cite{BenDaliDolega2023}]\label{thm Jch}
    For any partitions $\mu$ and $\lambda=[\lambda_1,\lambda_2,\dots,\lambda_s]$,
    $$\Jch_\mu(\lambda)=[t^{|\mu|}p_\mu]\exp\left(\Binf(-t,\bfp,-\alpha \lambda_1)\right)\dots\exp\left(\Binf(-t,\bfp,-\alpha \lambda_s)\right)\cdot 1.$$
where 
$$\Binf(t,\bfp,u):=\sum_{n\geq 1}\frac{t^n}{n}\B_n(\bfp,u).$$
\end{thm}

\subsection{Structure coefficients}\label{ssec structure coeff}
Jack characters form a linear basis of the space of shifted symmetric functions, see \cref{ssec shifted fcts}. Hence, their structure coefficients $g^\pi_{\mu,\nu}(\alpha)$ are well defined:
\begin{equation}\label{eq g Jch}
  \Jch_\mu\Jch_\nu=\sum_{\pi}g^\pi_{\mu,\nu}(\alpha)\Jch_\pi.  
\end{equation}

\medskip We now explain the connection between these coefficients and the enumeration of hypermaps. We introduce the series of the structure coefficients;
\begin{equation}\label{eq def G}
  G^{(\alpha)}(t,\bfp,\bfq,\bfr):=\sum_{\pi,\mu,\nu}\frac{g^\pi_{\mu,\nu}(\alpha)}{z_\pi \alpha^{\ell(\pi)}}t^{|\mu|+|\nu|-|\pi|}p_\pi q_\mu r_\nu.  
\end{equation}
\noindent Moreover, the coefficients $g^\pi_{\mu,\nu}$ satisfy the following property (see \cref{lem deg bound}),
\begin{equation}\label{eq upper bound}
  g^\pi_{\mu,\nu}=0 \text{ if }|\pi|>|\mu|+|\nu|.  
\end{equation}
Hence,  $G^{(\alpha)}$ is well defined in the algebra $\mcA:=\mathbb Q(\alpha)[t,\bfp]\llbracket \bfq,\bfr\rrbracket$. 
This is the space of formal power series in the variables  $q_i$ and $r_i$, whose coefficients are polynomial in $t$ and~$p_i$. 

The series $G^{(\alpha)}$ is actually an $\alpha$-deformation of the generating series of hypermaps with marked faces (see \cref{ssec proof prop H-G}).
\begin{prop}\label{prop H-G}
    For $\alpha \in \left\{1,2\right\}$, we have
    $$G^{(\alpha)}(t,\bfp,\bfq,\bfr)=\widetilde H^{(\alpha)}(t,\bfp,\bfq,\bfr).$$
\end{prop}


The hope is to generalize the previous proposition for any $\alpha$ by considering maps counted with $\alpha$-weights as in \cite{LaCroix2009,ChapuyDolega2022,BenDaliDolega2023}. However, such a question seems to be out of reach when we consider generating series with three alphabets $\bfp$, $\bfq$ and $\bfr$. The main goal of this paper is to shed some light on the combinatorics of the series $G^{(\alpha)}(t,\bfp,\bfq,\bfr)$ by giving a family of differential equations which characterize it.

\subsection{Main theorem}\label{ssec main thm}


 Let $\left(\Binf(t,\bfp,u)\right)^\perp$ denote the adjoint operator of $\Binf(t,\bfp,u)$ with respect to the usual scalar product of $\mathbb{Q}(\alpha)[\bfp]$, see \cref{ssec symmetric functions}. In \cref{ssec dual operators}, we provide a differential expression for this dual operator .

We are interested in the operators $\Binf(-t,\bfq,u)$, $\Binf(-t,\bfr,u)$ as well as $\left(\Binf(-t,\bfp,u)\right)^\perp$. These three operators are well defined as operators from  $\mcA$ to $\mcA\llbracket u\rrbracket$,  see \cref{rmq C_ell defined}. We now state the main theorem of the paper.

 \begin{thm}\label{thm diff eq}
The function $G^{(\alpha)}(t,\bfp,\bfq,\bfr)$ satisfies the following equation
\begin{equation}\label{eq diff eq}
  \left(\Binf(-t,\bfq,u)+\Binf(-t,\bfr,u)\right)\cdot G^{(\alpha)}(t,\bfp,\bfq,\bfr)= {\Binf}^{\perp}(-t,\bfp,u) \cdot G^{(\alpha)}(t,\bfp,\bfq,\bfr).  
\end{equation}
\end{thm}

The proof of this result is based on the differential construction of Jack characters given in \cref{thm Jch}. We also prove in \cref{prop characterization G} that \cref{eq diff eq} characterizes the series~$G^{(\alpha)}$.

By extracting coefficients in the variable $u$, \cref{eq diff eq} can be alternatively written as a family of equations (independent of $u$) which are indexed by non-negative integers; see \cref{ssec C l}.

Furthermore, we solve this differential equation and give an explicit expression of coefficients $g^\pi_{\mu,\nu}(\alpha)$ using some coefficients $a^\lambda_\mu$ which are obtained from the operator $\Binf$ and which are known to count maps (see \cref{thm g-a}).

It may be more convenient to think of the differential equation \cref{eq diff eq} as a commutation relation that we now explain. Let $\G(t,\bfp,\bfq,\bfr)$ be the operator defined by 
    \begin{alignat*}{3}
        \G(t,\bfp,\bfq,\bfr):\quad &\mathbb Q(\alpha)[ \bfp] \quad & \longrightarrow \quad& \mathbb Q(\alpha)[ \bfq,\bfr]\llbracket t\rrbracket&\\
            &p_\pi \quad&\longmapsto \quad &\sum_{\mu,\nu}t^{|\mu|+|\nu|-|\pi|}g^\pi_{\mu,\nu}(\alpha)q_\mu r_\nu.
    \end{alignat*}

 Actually,  $\G\cdot p_\pi$ is polynomial in $t$ of degree at most $|\pi|$ (see \cref{lem deg bound}). It turns out that for $\alpha=1$, $\G$ is a hypermap construction operator, which acts on a map by adding edges and coloring faces (see \cref{ssec op G}). The following is a variant of the main theorem (the equivalence between the two results follows from the definitions, see also \cref{lem g skew Jch}).
    \begin{thm}\label{thm com C-G}
    We have the following relation
    \begin{equation}\label{eq com B-G}
    \left(\Binf(-t,\bfq,u)+\Binf(-t,\bfr,u)\right)\cdot\G(t,\bfp,\bfq,\bfr)=\G(t,\bfp,\bfq,\bfr)\cdot \Binf(-t,\bfp,u)
    \end{equation}
    between operators from $\mathbb Q (\alpha)[\bfp]$ to $\mathbb Q(\alpha)[ \bfq,\bfr]\llbracket t,u\rrbracket$.
    \end{thm}

\noindent In \cref{sec alpha 1}, we give a combinatorial proof of this commutation relation for $\alpha=1$.

\subsection{Goulden--Jackson and \'{S}niady conjectures}
The Matching-Jack conjecture, introduced by Goulden and Jackson in \cite{GouldenJackson1996},  suggests that when $|\pi|=|\mu|=|\nu|$, coefficients $g^\pi_{\mu,\nu}$ satisfy  positivity and integrality  properties, and that they count hypermaps with "non-orientability weights"; see also \cref{conj MJ}. This conjecture is still open, despite many partial results \cite{DolegaFeray2016, ChapuyDolega2022,BenDali2022,BenDali2023}.

 The following conjecture due, to \'{S}niady, can be thought of as a generalization of the Matching-Jack conjecture to coefficients $g^\pi_{\mu,\nu}$ indexed by partitions of arbitrary sizes.

\begin{conj}$($\cite[Conjecture 2.2]{Sniady2019}$)$\label{conj g}
    For any $\pi,\mu$ and $\nu$ partitions, $g^\pi_{\mu,\nu}$ is a polynomial in $b:=\alpha-1$ with non-negative integer coefficients.
\end{conj}
\noindent
In \cite{DolegaFeray2016}, Do\l{}e\k{}ga and Féray have proved  that the coefficients   $g^\pi_{\mu,\nu}$ are polynomial in the deformation parameter $b$.  In \cref{sec g properties}, we deduce the integrality part in \cref{conj g} from a similar result for the coefficients of the Matchings-Jack conjecture, see \cref{cor integrality g}.

 Unfortunately, we have not been able to use the explicit expression of coefficients $g^\pi_{\mu,\nu}(\alpha)$ obtained in \cref{thm g-a} to prove their positivity in $b$ (the remaining part in  \cref{conj g}). It is however possible to use a variant of the main theorem to give a positive differential expression for the low degree terms of the operator $\G$.

\subsection{Low degree terms of \texorpdfstring{$\G$}{G}}
We consider the homogeneous parts of the operators $\G$ defined for any $k\geq 0$ by
\begin{equation}\label{eq:def:G_k}
  \G_k\cdot p_\pi=\sum_{|\mu|+|\nu|=|\pi|+k}g^\pi_{\mu,\nu}(\alpha)q_\mu r_\nu.  
\end{equation}

\noindent Note that from \cref{eq upper bound}, we have
$$\G(t,\bfp,\bfq,\bfr)=\sum_{k\geq 0}t^k\G_k(\bfp,\bfq,\bfr).$$
\noindent In \cref{ssec op G} we prove that for $\alpha=1$, the operator $\Gone_k$ is an operator which acts on a map by adding $k$ edges satisfying some conditions.

We use \cref{prop characterization G} (a variant of the main theorem) to give a differential expression for the operators $\G_0,\G_1$ and $\G_2$. First, we introduce the operator
\begin{alignat*}{3}
\Psi(\bfp,\bfq,\bfr):
\mathbb Q(\alpha)[ &\bfp]\quad &\longrightarrow \quad & \mathbb Q(\alpha)[\bfq,\bfr]\\
&p_\pi \quad &\longmapsto \quad& \prod_{1\leq i\leq \ell(\pi)}(q_{\pi_i}+r_{\pi_i}).
\end{alignat*}
Combinatorially, when we think of $p_i$ (resp. $q_i$, $r_i$) as the weight of an uncolored face (reps $(+)$ face, $(-)$ face) of degree $i$ in a map, $\Psi$ is the operator which chooses a color   $(+)$ or $(-)$ for each face.


\begin{thm}\label{thm low terms}
We have the following differential expressions for $\G_0,\G_1$ and $\G_2$.
\begin{equation}\label{eq G_0}
  \G_0=\Psi,  
\end{equation}

\begin{equation}\label{eq G_1}
  \G_1=\sum_{m\geq 1}\sum_{\genfrac..{0pt}{2}{m_1+m_2=m+1}{m_1,m_2\geq 1}}q_{m_1}r_{m_2}\cdot \Psi\cdot \frac{m\partial}{\partial p_m},  
\end{equation}

\noindent and
\begin{align}
  \G_2=&\frac{1}{2}\sum_{m\geq1}\sum_{m_1+m_2=m+2\atop{m_1,m_2\geq 1}}b(m_1-1)(m_2-1)q_{m_1}r_{m_2}\Psi\frac{m\partial}{\partial p_m}\label{eq G_2}\\ 
  &\nonumber+\frac{1}{2}\sum_{m\geq 1}\sum_{m_1+m_2+m_3=m+2\atop{m_1,m_2,m_3\geq 1}}(m_1-1)(q_{m_1}r_{m_2}r_{m_3}+r_{m_1}q_{m_2}q_{m_3})\Psi\frac{m\partial}{\partial p_m}\\\nonumber
  &+\frac{1}{2}\sum_{k,m\geq1}\sum_{i_1+i_2=k+m+2\atop{i_1,i_2\geq1}}\alpha\min(m,k,i_1-1,i_2-1)q_{i_1}r_{i_2}\Psi\frac{m\partial}{\partial p_m}\frac{k\partial}{\partial p_k}\\\nonumber
  &+\frac{1}{2}\sum_{m,k\geq1}\sum_{m_1+m_2=m+1\atop{m_1,m_2\geq1}}\sum_{k_1+k_2=k+1\atop{k_1,k_2\geq1}}q_{m_1}q_{k_1}r_{m_2}r_{k_2}\Psi
  \frac{m\partial}{\partial p_m}\frac{k\partial}{\partial p_k}.\nonumber
\end{align}
\end{thm}
This theorem can actually be obtained combinatorially from the cases $\alpha=1$, $\alpha=2$ and a polynomiality argument (all coefficients are polynomials of degree at most 1 in $\alpha$). However, this argument works only for $k\leq 2$ while the approach based on \cref{thm diff eq} which we present here can be used, with more computations, to understand the operators $\G_k$ for higher $k$.

\medskip Combining \cref{thm low terms} and \cref{cor integrality g}, we deduce the following special case of \cref{conj g}.
\begin{cor}
    Fix three  $\pi,\mu$ and $\nu$ partitions such that $|\pi|\geq |\mu|+|\nu|-2$. Then $g^\pi_{\mu,\nu}$ is a polynomial in $b:=\alpha-1$ with non-negative integer coefficients.
\end{cor}

We hope that a better understanding of the differential structure of the operator $\Binf$ could allow one to generalize \cref{thm low terms} in order to obtain a differential formula of $\G_k$ for any $k$. This would eventually give the missing positivity  part in \cref{conj g} and in Goulden--Jackson's Matching-Jack conjecture.

\subsection{\texorpdfstring{$k$}{k}-constellations and Hurwitz numbers}
\textit{$k$-constellations} represent a family of maps whose vertices are colored in $k+1$ colors and from which hypermaps can be obtained by setting $k=1$. In the orientable case, $k$-constellations are related to the factorizations of the identity in the symmetric group into $k+2$ permutations \cite{BousquetSchaeffer2000}. Recently, a model of constellations on non-orientable surfaces have been introduced in \cite{ChapuyDolega2022}. In these two cases, generating series of constellations have been a useful tool to understand \textit{Hurwitz numbers} by considering some particular specializations (see also \cite{BonzomChapuyDolega2022,BonzomNador2023}).

Moreover, orientable constellations with control of all color types are in bijection with the ramified coverings of the sphere above an arbitrary number of points with the  full ramification profiles; see \cite[Section 1.2]{LandoZvonkin2004} (see also \cite[Section 2.2]{ChapuyDolega2022} for the non-orientable case).

The approach used here to prove the main theorem applies to the case of constellations and allows us to extend the differential equation to series with $k$ alphabets; see \cref{thm constellations}.

\subsection{Equations for connected series}
In this paper, we are considering generating series of  hypermaps not necessarily connected. Nevertheless, it is possible to obtain the generating series of connected hypermaps by taking a logarithm. More precisely, the series 
$$\hG^{(\alpha)}(t,\bfp,\bfq,\bfr)=\alpha\cdot\log(G^{(\alpha)}(t,\bfp,\bfq,\bfr))$$
is an $\alpha$-deformation of the generating series of connected hypermaps with marked faces. The homogeneous part of this series is the object of Goulden--Jackson's $b$-conjecture (known also as the hypermap-Jack conjecture) \cite[Conjecture 6.3]{GouldenJackson1996}. 

In \cref{sec connected}, we derive from the main theorem a differential equation for $\hG^{(\alpha)}$, see \cref{thm diff for connected}.

\subsection{Outline of the paper}
The paper is organized as follows. In \cref{sec preliminaries}, we give some preliminaries related to partitions, symmetric and shifted symmetric functions. In \cref{sec g properties}, we establish some useful properties of structure coefficients $g^\pi_{\mu,\nu}$, we discuss their connection to the Matching-Jack conjecture and we prove \cref{prop H-G}. \cref{sec main thm} is dedicated to the proof of the main theorem as well as its generalized version \cref{thm constellations} related to Hurwitz numbers. In \cref{sec alpha 1}, we give a combinatorial proof of \cref{thm com C-G} for $\alpha=1$ (this section is quite independent from the rest of the paper).  We use the main theorem in \cref{sec resolution} to give an explicit expression for coefficients $g^\pi_{\mu,\nu}$ and to prove \cref{thm low terms}.  Finally, we prove \cref{thm diff for connected} in \cref{sec connected}.

\section{Preliminaries}\label{sec preliminaries}

\subsection{Partitions}\label{subsec Partitions}

A \textit{partition} $\lambda=[\lambda_1,...,\lambda_s]$ is a weakly decreasing sequence of positive integers $\lambda_1\geq...\geq\lambda_s>0$. We denote by $\mathbbm{Y}$ the set of all integer partitions. The integer $s$ is called the \textit{length} of $\lambda$ and is denoted $\ell(\lambda)$. The size of $\lambda$ is the integer $|\lambda|:=\lambda_1+\lambda_2+...+\lambda_s.$ If $n$ is the \textit{size} of $\lambda$, we say that $\lambda$ is a partition of $n$ and we write $\lambda\vdash n$. The integers $\lambda_1$,...,$\lambda_s$ are called the \textit{parts} of $\lambda$. For $i\geq 1$, we denote $m_i(\lambda)$ the number of parts of size $i$ in $\lambda$. We then set 
$$z_\lambda:=\prod_{i\geq1}m_i(\lambda)!i^{m_i(\lambda)}.$$ 
We denote by $\leq$ the \textit{dominance partial} order on partitions, defined by 
$$\mu\leq\lambda \iff |\mu|=|\lambda| \text{ and }\hspace{0.3cm} \mu_1+...+\mu_i\leq \lambda_1+...+\lambda_i \text{ for } i\geq1.$$

\noindent Finally, we identify a partition  $\lambda$ with its \textit{Young diagram}, defined by 
$$\lambda:=\{(i,j),1\leq i\leq \ell(\lambda),1\leq j\leq \lambda_i\}.$$



\subsection{Symmetric functions}\label{ssec symmetric functions}
We fix an alphabet $\mathbf{x}:=(x_1,x_2,..)$. We denote by $\mathcal{S}_\alpha$ the algebra of symmetric functions in $\mathbf{x}$ with coefficients in $\mathbb Q(\alpha)$. For every partition $\lambda$, we denote $m_\lambda$ the monomial function 
$$m_\lambda(\bfx):=\sum_{\beta=(\beta_1,\dots,\beta_s)}\sum_{1 \leq i_1\leq\dots \leq i_s}x_{i_1}^{\beta_1}\dots x_{i_s}^{\beta_s},$$
where the sum is taken over all reorderings $\beta$ of the partition $\lambda$. Moreover, let $p_\lambda$ denote the power-sum symmetric function, defined as follows; if $n\geq 1$ then 
$$p_n(\bfx):=\sum_{i\geq 1}x_i^n,$$ and if $\lambda=[\lambda_1,\dots,\lambda_s]$ then 
$$p_\lambda(\bfx)=p_{\lambda_1}(\bfx)\dots p_{\lambda_s}(\bfx).$$
We  consider the associated alphabet of power-sum functions
$\mathbf{p}:=(p_1,p_2,..)$. 
It is well known that monomial functions and the
power-sum functions both form bases of the symmetric function algebra;
therefore $\mathcal{S}_\alpha$ can be identified with the polynomial algebra $\mcP:=\Span_{\mathbb{Q}(\alpha)}\{p_\lambda\}_{\lambda\in\mathbb{Y}}$. If $f$ is a symmetric function in the alphabet $\bfx$, it will be convenient to denote with the same letter the function and the associated polynomial in $\bfp$;
$$f(\bfx)\equiv f(\bfp).$$

We denote by $\langle.,.\rangle_\alpha$ the $\alpha$-deformation of the Hall scalar product defined on $\mcS$ by 
$$\langle p_\lambda,p_\mu\rangle_\alpha=z_\lambda\alpha^{\ell(\lambda)}\delta_{\lambda,\mu},\text{ for any partitions }\lambda,\mu,$$
where $\delta_{\lambda,\mu}$ denotes the Kronecker delta.

Macdonald has established the following characterization theorem for Jack polynomials which we take as a definition; see \cite[Chapter VI, Section 10]{Macdonald1995}.
\begin{thm}[\cite{Macdonald1995}]\label{thm Jack polynomials}
    \textit{Jack polynomials} $(J_\lambda^{(\alpha)})_{\lambda\in\mathbb Y}$ are the unique family of symmetric functions  in $\mathcal{S}_\alpha$ indexed by partitions, satisfying the following properties:
    \begin{itemize}
        \item Orthogonality:  
        $$\langle J_\lambda^{(\alpha)},J_\mu^{(\alpha)}\rangle_\alpha=0, \text{ for }\lambda\neq\mu.$$
        \item Triangularity:
     $$[m_\mu]J_\lambda^{(\alpha)}=0, \text{ unless }\mu\leq\lambda.$$
     \item Normalization:
     \begin{equation}\label{eq normalization}
       [p_{1^n}]J_\lambda^{(\alpha)}=1, \text{ for }\lambda\vdash n,  
     \end{equation}
     where $1^n$ is the partition with $n$ parts equal to~1. 
    \end{itemize}
Moreover, Jack polynomials form a basis of $\mcS$. 
\end{thm}

We denote by $j_\lambda^{(\alpha)}$ the squared-norm of $J_\lambda^{(\alpha)}$;
\begin{equation}\label{eq j alpha}
    j_\lambda^{(\alpha)}:=\langle J_\lambda^{(\alpha)} ,J_\lambda^{(\alpha)}\rangle_\alpha.
\end{equation}

\begin{rmq}\label{rmq empty g}
Let $\Jch_\emptyset$ denote the Jack character indexed by the empty partition. 
It follows from \cref{eq normalization} that 
    $\Jch_\emptyset(\lambda)=1$ for any $\lambda\in\mathbb Y$. Hence, $g^\emptyset_{\emptyset,\emptyset}(\alpha)=1$.
\end{rmq}

\subsection{Jack characters and shifted symmetric functions}\label{ssec shifted fcts}
\begin{defi}[\cite{Lassalle2008b}]
We say that a polynomial of degree $n$ in $k$ variables
$(s_1,\dots,s_k)$ with coefficients in $\mathbb Q (\alpha)$ is
\textit{$\alpha$-shifted symmetric} if it is symmetric in the variables $s_i-i/\alpha$.
An $\alpha$-shifted symmetric function (or simply a shifted symmetric function)  is a sequence
$(f_k)_{k\geq1}$ of shifted symmetric polynomials of bounded degrees, such that for every $k\geq 1$, the function $f_k$ is
an $\alpha$-shifted symmetric polynomial in $k$ variables and 
\begin{equation}\label{eq shifted functions}
  f_{k+1}(s_1,\dots,s_k,0)=f_k(s_1,\dots,s_k).  
\end{equation}
We denote by $\mcSstar$ the algebra of shifted symmetric functions.
\end{defi}
Let $f$ be a shifted symmetric function and let $\lambda =
(\lambda_1,\dots,\lambda_k)$ be a partition. Then we denote
$f(\lambda):=f(\lambda_1,\dots, \lambda_{k},0,\dots)$. 

\begin{thm}[\cite{KnopSahi1996}]\label{thm Knop Sahi}
    Let $n\geq 0$. And let $g$ be a function on Young diagrams. There exists a unique shifted symmetric function $f$ of degree less or equal than $n$ such that $f(\lambda)=g(\lambda)$ for any $|\lambda|\leq n$.
\end{thm}

In particular, a shifted symmetric function is completely determined by its evaluation on Young diagrams $(f(\lambda))_{\lambda\in\mathbb{Y}}$. The following theorem due to Féray gives a characterization of Jack characters as shifted symmetric functions satisfying some properties (see \cite{BenDaliDolega2023} for a proof).

\begin{thm}[Féray]\label{thm Feray}
Fix a partition $\mu$. The Jack character $\Jch_\mu$ is the unique $\alpha$-shifted symmetric function of degree $|\mu|$ with top homogeneous part $\alpha^{|\mu|-\ell(\mu)}/z_\mu\cdot p_\mu$, such that $\Jch_\mu(\lambda)=0$ for any partition $|\lambda|<|\mu|$.
\end{thm}

\section{Structure coefficients \texorpdfstring{$g^\pi_{\mu,\nu}(\alpha)$}{g} and proof of Proposition \texorpdfstring{\ref{prop H-G}}{}}\label{sec g properties}
The purpose of this section is to discuss some properties of coefficients $g^\pi_{\mu,\nu}$. In particular, we use their connection with the coefficients of the Matchings-Jack conjecture to establish integrality in \cref{conj g} and to prove \cref{prop H-G}.
\subsection{Some properties}
We start by proving some properties of the structure coefficients~$g^\pi_{\mu,\nu}$.
\begin{lem}\label{lem deg bound}
    The coefficient $g^\pi_{\mu,\nu}(\alpha)$ is $0$ unless $\max(|\mu|,|\nu|)\leq |\pi|\leq |\mu|+|\nu|.$
    \begin{proof} 
    The upper bound is a direct consequence of the fact that $\Jch_\mu$ is a shifted symmetric function of degree $|\mu|$ and the fact that $(\Jch_{\pi})_{|\pi|\leq d}$ is a basis of shifted symmetric functions of degree less or equal than $d$, see \cref{thm Feray}. In order to obtain the lower bound we use the vanishing properties of $\Jch_\mu$. Fix two partitions $\mu$ and $\nu$. Set $m:=\max(|\mu|,|\nu|)$ and 
    \begin{align}
      F &:=\Jch_\mu\Jch_\nu-\sum_{m\leq |\pi|\leq |\mu|+|\nu|}g^\pi_{\mu,\nu}\Jch_\pi\label{eq lem deg bound 1}\\
      &=\sum_{|\pi|< m}g^\pi_{\mu,\nu}\Jch_\pi.\label{eq lem deg bound 2}  
    \end{align}
    From \cref{eq lem deg bound 2}, the function $F$ is shifted symmetric with degree at most $m-1$. Moreover, using \cref{eq lem deg bound 1} and the definition of Jack characters, we get that $F(\lambda)=0$ for any $|\lambda|<m$. Applying \cref{thm Knop Sahi}, we deduce that $F=0$. By consequence $g^\pi_{\mu,\nu}=0$ for any $|\pi|<m$.
    \end{proof}
\end{lem}

As a consequence of this lemma, we get that the series $G^{(\alpha)}$ introduced in \cref{eq def G} is well defined in $\mathbb Q(\alpha)[t,\bfp]\llbracket \bfq,\bfr\rrbracket\cap \mathbb Q(\alpha)[t,\bfq,\bfr]\llbracket \bfp\rrbracket$.

\subsection{Goulden--Jackson's Matchings-Jack conjecture}\label{ssec MJ}
Goulden and Jackson have introduced in \cite{GouldenJackson1996} the coefficients $c^\pi_{\mu,\nu}(\alpha)$ indexed by three partitions of the same size, and defined by the following expansion:
\begin{equation}\label{eq:tau}
  \tau^{(\alpha)}(t,\mathbf{p},\mathbf{q},\mathbf{r}):=\sum_{n\geq0}t^n\sum_{\theta\vdash n}\frac{1}{j^{(\alpha)}_\theta}J^{(\alpha)}_\theta(\mathbf{p})J^{(\alpha)}_\theta(\mathbf{q})J^{(\alpha)}_\theta(\mathbf{r})=\sum_{n\geq 0}t^n\sum_{\pi,\mu,\nu\vdash n}\frac{c^\pi_{\mu,\nu}(\alpha)}{z_\pi\alpha^{\ell(\pi)}}p_\pi q_\mu r_\nu.  
\end{equation}

For $\alpha\in\{1,2\}$, this series is known to be the generating series of hypermaps. 
\begin{thm}[\cite{GouldenJackson1996a}]\label{thm coef c b=0}
Fix three partitions $\pi$, $\mu$ and $\nu$ of the same size.
For $\alpha=1$ (resp. $\alpha=2$), the coefficient $c^\pi_{\mu,\nu}(1)$ (resp. $c^\pi_{\mu,\nu}(2)$) counts the number of oriented (resp. orientable or not) vertex labelled hypermaps of profile $(\pi,\mu,\nu)$.
Equivalently, for $\alpha\in\left\{1,2\right\}$ we have
$$\tau^{(\alpha)}(t,\mathbf{p},\mathbf{q},\mathbf{r})=H^{(\alpha)}(t,\mathbf{p},\mathbf{q},\mathbf{r}),$$
\end{thm}
\noindent These two cases are obtained using representation theory tools (the case $\alpha=1$ is a classical result, while the case $\alpha=2$ is due to Goulden and Jackson \cite{GouldenJackson1996a}).

The Matching-Jack conjecture states that the coefficients $c^\pi_{\mu,\nu}(\alpha)$ still have a combinatorial interpretation in terms of maps for any $\alpha$.
\begin{conj}[\cite{GouldenJackson1996}]\label{conj MJ}
For any partitions  $\pi$, $\mu$ and $\nu$ of the same size, $c^\pi_{\mu,\nu}(\alpha)$ is a polynomial in the shifted parameter $b:=\alpha-1$ with non-negative integer coefficients. Equivalently, there exists a statistic  $\vartheta$ on non-oriented hypermaps with non-negative integer values, such that  
\begin{itemize}
    \item $\vartheta(M)=0$ if and only if $M$ is oriented,
    \item for any partitions $\pi$, $\mu$ and $\nu$ of the same size
    \begin{equation*}
  c^\pi_{\mu,\nu}=\sum_{M\in\hypermap ^\pi_{\mu,\nu}} b^{\vartheta(M)}.
\end{equation*}
\end{itemize}
\end{conj}

 In the next subsection, we use the following integrality result for coefficients $c^\pi_{\mu,\nu}$ in order to obtain a similar result for coefficients $g^\pi_{\mu,\nu}$.
\begin{thm}[\cite{BenDali2023}]\label{thm integrality c}
    The coefficients $c^\pi_{\mu,\nu}$ are polynomials in $b=\alpha-1$ with integer coefficients.
\end{thm}

\subsection{Links between coefficients \texorpdfstring{$c^\pi_{\mu,\nu}$}{c} and \texorpdfstring{$g^\pi_{\mu,\nu}$}{g}}
The following proposition has been proved by Do\l{}e\k{}ga and Féray \cite[Proposition B.1]{DolegaFeray2016}.
\begin{prop}\label{prop_c-g_1}
If $\pi,\mu$ and $\nu$ are of the same size then 
  $$c^\pi_{\mu,\nu}(\alpha)=g^\pi_{\mu,\nu}(\alpha).$$  
\end{prop}
\noindent As a consequence, \cref{conj g} is a generalization of \cref{conj MJ}. 

If $\pi$ is a partition, we denote by $\widetilde \pi:=\pi\backslash 1^{m_1(\pi)}$ the partition obtained by deleting all parts of size 1. The following proposition is a generalization of \cite[Equation (19)]{DolegaFeray2016}. The proof is quite the same. 
\begin{prop}\label{prop_c-g_2}
For every partitions $\pi,\mu$ and $\nu$ such that $\pi\vdash n\geq |\mu|,|\nu|$, we have
\begin{equation}\label{eq c-g}
  \sum_{i=0}^{m_1(\pi)}\binom{m_1(\pi)}{i}g^{\tilde{\pi}\cup1^i}_{\mu,\nu}=\binom{m_1(\mu)+n-|\mu|}{m_1(\mu)}\binom{m_1(\nu)+n-|\nu|}{m_1(\nu)}c^\pi_{\mu\cup1^{n-|\mu|},\nu\cup1^{n-|\nu|}},  
\end{equation}
Equivalently, 
\begin{equation}\label{eq c-g 2}
  \exp\left(\frac{p_1}{t\alpha}\right)G^{(\alpha)}(t,\bfp,\bfq,\bfr)=\exp\left(\frac{\partial}{t\partial q_1}+\frac{\partial}{t
  \partial r_1}\right)\tau^{(\alpha)}(t,\bfp,\bfq,\bfr),
\end{equation}
where the last equality holds in $\mathbb Q(\alpha) [t,1/t,\bfq,\bfr] \llbracket \bfp\rrbracket.$

\begin{proof}
     To simplify expressions, we denote $\bar\mu:=\mu\cup1^{n-|\mu|}$ and $\bar \nu:=\nu\cup1^{n-|\nu|}$. We get from the definition of Jack characters and \cref{lem deg bound} that for any partition $\lambda$ of size $n$
    \begin{align*}
      \Jch_\mu(\lambda)\Jch_\nu(\lambda)
      &=\binom{n-|\mu|+m_1(\mu)}{m_1(\mu)}\binom{n-|\nu|+m_1(\nu)}{m_1(\nu)}
      \Jch_{\bar\mu}(\lambda)
      \Jch_{\bar\nu}(\lambda)\\
      &=\binom{n-|\mu|+m_1(\mu)}{m_1(\mu)}\binom{n-|\nu|+m_1(\nu)}{m_1(\nu)}\sum_{n\leq |\rho|\leq 2n}g^\rho_{\bar\mu,\bar\nu}(\alpha)\Jch_\rho(\lambda).
    \end{align*}

    Terms corresponding to $|\rho|>n$ vanish since characters are evaluated at a partition of size $n$. Hence
    \begin{align*}
      \Jch_\mu(\lambda)\Jch_\nu(\lambda)
      =\binom{n-|\mu|+m_1(\mu)}{m_1(\mu)}\binom{n-|\nu|+m_1(\nu)}{m_1(\nu)}\sum_{\rho\vdash n}g^\rho_{\bar\mu,\bar\nu}(\alpha)\Jch_\rho(\lambda).
    \end{align*}

    Using \cref{prop_c-g_1} we get that,
    \begin{equation}\label{eq 1 prop c-g}
      \Jch_\mu(\lambda)\Jch_\nu(\lambda)=\binom{n-|\mu|+m_1(\mu)}{m_1(\mu)}\binom{n-|\nu|+m_1(\nu)}{m_1(\nu)}\sum_{\rho\vdash n}c^\rho_{\bar\mu,\bar\nu}\Jch_\rho(\lambda).  
    \end{equation}

On the other hand, for any $\lambda\vdash n$,
\begin{align}
\Jch_\mu(\lambda)\Jch_\nu(\lambda)
&=\sum_{|\kappa|\leq n}g^\kappa_{\mu,\nu}(\alpha)\Jch_\kappa(\lambda) \nonumber \\
&=\sum_{|\kappa|\leq n}g^\kappa_{\mu,\nu}(\alpha)\binom{m_1(\kappa)+n-|\kappa|}{m_1(\kappa)}\Jch_{\kappa\cup1^{n-|\kappa|}}(\lambda)\nonumber\\
&=\sum_{\rho\vdash n}\Jch_\rho(\lambda)\left(\sum_{i=0}^{m_1(\rho)}g^{\widetilde \rho\cup 1^i}_{\mu,\nu}(\alpha)\binom{m_1(\rho)}{i}\right).\label{eq 2 prop c-g}
\end{align}
The last equation is obtained by regrouping terms with respect to $\rho:=\kappa\cup 1^{n-|\kappa|}.$
We obtain \cref{eq c-g} by comparing the coefficient of $\Jch_\pi$ in Eqs. \eqref{eq 1 prop c-g} and \eqref{eq 2 prop c-g}. 

Let us now prove  \cref{eq c-g 2}. Let $\pi$, $\mu$ and $\nu$ be three partitions. We want to prove that the coefficient of $t^{|\mu|+|\nu|-|\pi|}p_\pi q_\mu r_\nu/(z_\pi \alpha^{\ell(\pi)})$ is the same in both sides of \cref{eq c-g 2}. It is easy to check that this is given by \cref{eq c-g} if $|\pi|\geq \max(|\mu|,|\nu|)$. Otherwise, each one of these coefficients is 0; 
this is a consequence of \cref{lem deg bound} and the fact that $\tau^{(\alpha)}(t,\bfp,\bfq,\bfr)$ is homogeneous in the three alphabets $\bfp$, $\bfq$ and $\bfr$.
\end{proof}
\end{prop}

We deduce the following corollary.
\begin{cor}\label{cor integrality g}
The coefficients $g^\pi_{\mu,\nu}$ are polynomials in $b$ with integer coefficients.
\begin{proof}
Inverting \cref{eq c-g 2}, we get
\begin{equation}\label{eq G-tau}
  G^{(\alpha)}(t,\bfp,\bfq,\bfr)=\exp\left(-\frac{p_1}{t\alpha}\right)\exp\left(\frac{\partial}{t\partial q_1}+\frac{\partial}{t\partial r_1}\right)\tau^{(\alpha)}(t,\bfp,\bfq,\bfr).  
\end{equation}
By extracting the coefficient of $t^{|\mu|+|\nu|-|\pi|}p_\pi q_\mu r_\nu/(z_\pi \alpha^{\ell(\pi)})$, we get
\begin{multline*}
  g^\pi_{\mu,\nu}=\sum_{e\leq i\leq m_1(\pi)}(-1)^{m_1(\pi)-i}\binom{m_1(\pi)}{i}\binom{|\widetilde \pi|+i-|\mu|+m_1(\mu)}{m_1(\mu)}\\
  \binom{|\widetilde \pi|+i-|\nu|+m_1(\nu)}{m_1(\nu)}c^{\widetilde \pi \cup 1^{i}}_{\mu\cup 1^{|\widetilde\pi|+i-|\mu|},\nu\cup 1^{|\widetilde\pi|+i-|\nu|}},
\end{multline*}
where $\widetilde \pi:=\pi\backslash 1^{m_1(\pi)}$
and 
$e:=\max(0,|\mu|-|\widetilde \pi|,|\nu|-|\widetilde \pi|).$
We conclude using \cref{thm integrality c}.
\end{proof}
\end{cor}

\subsection{Proof of Proposition \texorpdfstring{\ref{prop H-G}}{}}\label{ssec proof prop H-G}
\begin{proof}[Proof of Proposition \texorpdfstring{\ref{prop H-G}}{}]
We prove the proposition for $\alpha=1$. The proof is exactly the same for $\alpha=2$. Fix three partitions $\pi,\mu$ and $\nu$. We want to prove that
\begin{equation}\label{eq prop H-G}
    g^\pi_{\mu,\nu}(1)=\left|\ohypermap^\pi_{\mu,\nu}\right|,
\end{equation}
where $\ohypermap^\pi_{\mu,\nu}$ is the set of hypermaps defined in \cref{def hypermaps}.
Since \cref{eq c-g} fully characterizes the coefficients $g^\pi_{\mu,\nu}$, it is enough to prove that
\begin{equation*}\label{eq tilde g c}
  \sum_{i=0}^{m_1(\pi)}\binom{m_1(\pi)}{i}\left|\ohypermap^{\tilde{\pi}\cup1^i}_{\mu,\nu}\right|=\binom{m_1(\mu)+n-|\mu|}{m_1(\mu)}\binom{m_1(\nu)+n-|\nu|}{m_1(\nu)}c^\pi_{\mu\cup1^{n-|\mu|},\nu\cup1^{n-|\nu|}}(1),  
\end{equation*}
holds for any partitions $\pi$, $\mu$ and $\nu$, with $n=|\pi|$. Using \cref{thm coef c b=0}, we know that the right-hand side of the last equation counts vertex labelled oriented hypermaps of profile $(\pi,\mu\cup1^{|\pi|-|\mu|},\nu\cup1^{|\pi|-|\nu|})$, with $|\pi|-|\mu|$ marked $(+)$ faces of degree 1 and $|\pi|-|\nu|$ marked $(-)$ faces of degree 1 (unlike in \cref{def hypermaps}, it is possible here to have both faces of isolated loops marked). 

On the other hand, the set of such maps $M$ with a fixed number $j$ of isolated loops with both $(+)$ and $(-)$ faces marked, can be obtained as follows:
\begin{itemize}
    \item choose the labels of black vertices of degree 2, which form the isolated loops with two marked faces; there are $\binom{m_1(\pi)}{j}$ such possible choices,
    \item choose a hypermap in $\ohypermap^{\widetilde\pi\cup1^{|\pi|-j}}_{\mu,\nu}$ and associate to black vertices of degree 2 the labels not chosen in the first step.
\end{itemize}
Summing over all $i:=m_1(\pi)-j$ between 0 and $m_1(\pi)$, we obtain the left hand-side. This finishes the proof of the proposition.
\end{proof}

We conclude this section by the following table which summarizes the different results proved or recalled in this section.

\begin{center}
    \setlength\arrayrulewidth{0.75pt}
    \begin{tabular}{|c|c|c|}
        \hline
        \parbox[][1.5cm][c]{5cm}{\centering \textcolor{red}{Series}}& \textcolor{red}{$\tau^{(\alpha)}$ (homogeneous)} & \textcolor{red}{$G^{(\alpha)}$ (non-homogeneous)}\\
       \hline
       \parbox[][1.5cm][c]{5cm}{\centering Combinatorial interpretation for $\alpha\in\{1,2\}$}
       & \parbox[][1.5cm][c]{4.5cm}{\centering{\cite{GouldenJackson1996} (see also \cref{thm coef c b=0}})}& 
       \parbox[][1.5cm][c]{5cm}{\centering{\cref{prop H-G} (see also \cref{prop g0-pre hyper} for $\alpha=1$)}}\\
       \hline
       \parbox[][1.5cm][c]{5cm}{\centering 
       Relations between the series}&
       \multicolumn{2}{c|}{Propositions \ref{prop_c-g_1} and \ref{prop_c-g_2}}\\
        \hline 
       \parbox[][1.5cm][c]{5cm}{\centering 
        Integrality of the coefficients (\cref{conj MJ})}& \parbox[][1.5cm][c]{4.5cm}{\centering{\cite{BenDali2023} (see also \cref{thm integrality c})}}&
       \cref{cor integrality g}\\
       \hline
    \end{tabular}
\end{center}





\section{Proof of the main theorem}\label{sec main thm}

\subsection{Differential operators}\label{ssec op Bn}
The purpose of this subsection is to recall the definition of operators $\B_n$ from \cite{ChapuyDolega2022}. These operators are defined using an extra catalytic alphabet $Y:=(y_i)_{i\geq 0}$. We recall that $\mcP:=\Span_{\mathbb{Q}(\alpha)}\{p_\lambda\}_{\lambda\in\mathbb{Y}}=\mathbb Q(\alpha)[\bfp]$.
We also consider the spaces
$$\PY:=\Span_{\mathbb{Q}(\alpha)}\left\{y_ip _\lambda\right\}_{i\in \mathbb{N},\lambda\in\mathbb{Y}}\text{ and }\quad \tildePY:=\mcP\oplus \PY.$$
We recall that $b:=\alpha-1$.
 Set the catalytic operators;
\begin{equation*}
Y_+=\sum_{i\geq 1}y_{i+1}\frac{\partial}{\partial y_i}: \mcP_Y\rightarrow\PY,
\end{equation*}
\begin{equation*}
\GY=(1+b)\cdot\sum_{i,j\geq1}y_{i+j}\frac{j\partial^2}{\partial y_{i-1}\partial p_{j}}+\sum_{i,j\geq1} y_{i}p_j\frac{\partial}{\partial y_{i+j-1}} +b\cdot \sum_{i\geq1}y_{i+1}\frac{i\partial}{\partial y_i}: \PY\rightarrow\PY,
\end{equation*}
$$\text{and }\quad \Theta_Y:=\sum_{i\geq1}p_i\frac{\partial}{\partial y_i}: \PY\rightarrow\mcP.$$

For $n\geq0$, and $u$ a variable, the operator $\mathcal B^{(\alpha)}_n(\bfp,u)$ is defined by
\begin{equation}\label{eq def Bn}
  \B_n(\bfp,u):=\Theta_Y\left(\GY+u Y_+\right)^{n}\frac{y_0}{1+b}\colon \mcP \rightarrow \mcP[u]
\end{equation}
and  the operator $\Binf$ by 
$$\Binf(t,\bfp,u):=\sum_{n\geq 1}\frac{t^n}{n}\B_n(\bfp,u):\mcP\rightarrow \mathbb \mcP\llbracket u,t\rrbracket.$$


\begin{rmq}\label{rmq C_ell defined}
    Note that the operator $\B_n$ is homogeneous of degree $n$; namely for any $\lambda$, we have that $\B_n\cdot p_\lambda$ is a linear combination of $u^\ell p_{\mu}$ for $\mu$ of size $|\lambda|+n$ and $\ell\leq n$. Similarly, if $n\leq |\lambda|$, then $\left({\B_n}\right)^\perp \cdot p_\lambda$ is a linear combination of $u^\ell p_{\mu}$ for $\mu$ of size $|\lambda|-n$ and $\ell\leq n$.

    Hence, $\Binf(t,\bfp,u)\cdot p_{\lambda}$ is a combination of $t^{|\mu|-|\lambda|}u^\ell p_{\mu}$ for $|\mu|\geq |\lambda|$ and $\ell\geq 0$.  Moreover,  ${\Binf}^\perp(t,\bfp,u)\cdot p_{\lambda}$ is a linear combination of $t^{|\lambda|-|\mu|}p_{\mu}u^\ell$ for $|\mu|\leq |\lambda|$ and $\ell\geq 0$. Consequently, operators 
    $$\Binf(t,\bfp,u): \mathbb Q(\alpha)[ t]\llbracket \bfp\rrbracket \longrightarrow \mathbb Q(\alpha)[ t]\llbracket \bfp,u\rrbracket  $$    
    $${\Binf}^\perp(t,\bfp,u): \mathbb Q(\alpha)[t,\bfp] \longrightarrow\mathbb Q(\alpha)[t,\bfp]\llbracket u\rrbracket  $$
    are well defined.
    We deduce that operators $\Binf(t,\bfq,u), \Binf(t,\bfr,u)$ and ${\Binf}^\perp(t,\bfp,u)$ are well defined from $\mcA$ to $\mcA\llbracket u\rrbracket$.   
\end{rmq}





\subsection{Skew Jack characters}
Before proving \cref{thm diff eq}, we introduce a skew\footnote{This is not the usual definition of skew characters in which we consider skew diagrams in the argument of the character; $\Jch_\mu(\lambda/ \rho)$.} version of Jack characters.
\begin{defi}
We consider a sequence of variables $u_1,u_2 \dots$.
For any partitions $\mu$ and $\nu$ satisfying $|\mu|\geq|\nu|$, we define the coefficient $\Jch_{\mu/\nu}$ which depends on one variable~$v$, by
$$\Jch_\mu(v,u_1,u_2\dots)=\sum_{\nu}\Jch_{\mu/\nu}(v)\Jch_\nu(u_1,u_2\dots).$$
This expansion is well defined, since $\Jch_\mu(v,u_1,u_2\dots)$ is a shifted symmetric function in $u_1,u_2\dots$, and $(\Jch_\nu)_{\nu\in\mathbb Y}$ is a basis of $\mcSstar$.
\end{defi}

We then have the following lemma.
\begin{lem}\label{lem g skew Jch}
For any partitions $\mu,\nu,\pi$ one has
$$\sum_{\kappa}g^\kappa_{\mu,\nu}\Jch_{\kappa/\pi}(v)=\sum_{\rho,\xi}g^\pi_{\rho,\xi}\Jch_{\mu/\rho}(v)\Jch_{\nu/\xi}(v).$$
\begin{proof}
 We have 
\begin{align*}
    \sum_{\pi}\Jch_\pi(u_1,u_2\dots) \left(\sum_{\kappa} g^\kappa_{\mu,\nu} \Jch_{\kappa/ \pi}(v)\right)
    &=\sum_{\kappa}g^{\kappa}_{\mu,\nu} \Jch_\kappa(v,u_1,u_2\dots)\\
    &=\Jch_{\mu}(v,u_1,u_2\dots)\Jch_{\nu}(v,u_1,u_2\dots)\\
    &=\sum_{\rho,\xi}\Jch_{\mu/\rho}(v)\Jch_\rho(u_1,u_2\dots)\Jch_{\nu/\xi}(v)\Jch_\xi(u_1,u_2\dots)\\
    &=\sum_{\pi} \Jch_\pi(u_1,u_2\dots)\left(\sum_{\rho,\xi}g^\pi_{\rho,\xi}\Jch_{\mu/\rho}(v)\Jch_{\nu/\xi}(v)\right).
\end{align*}
We conclude by extracting the coefficient of $\Jch_\pi(u_1,u_2,\dots).$
\end{proof}
\end{lem}

The following proposition gives a differential construction for skew characters.
\begin{prop}
    For any partitions $\mu$ and $\nu$ one has
    \begin{equation}\label{eq 1 def skew Jch}
  \Jch_{\mu/\nu}(v)=[t^{|\mu|-|\nu|}p_{\mu}]\exp\left(\B_\infty(-t,\bfp,-\alpha v)\right)\cdot p_\nu,  
\end{equation}
\begin{equation}\label{eq 2 def skew Jch}
\hspace{-0.5cm}\text{and }\quad \frac{\alpha^{\ell(\mu)}z_\mu}{\alpha^{\ell(\nu)}z_\nu}\Jch_{\mu/\nu}(v)=[t^{|\mu|-|\nu|}p_{\nu}]\exp\left({\Binf}^\perp(-t,\bfp,-\alpha v)\right)\cdot p_\mu.
\end{equation}
\begin{proof}
Fix $k\geq 0$. From \cref{thm Jch}, we have
$$\sum_{\nu}t^{|\nu|}\Jch_\nu(u_1,\dots,u_k) p_\nu= \exp\left(\Binf(-t,\bfp,-\alpha u_1)\right)\dots\exp\left(\Binf\left(-t,\bfp,-\alpha u_k\right)\right)\cdot 1.$$
By applying $\exp\left(\Binf(-t,\bfp,-\alpha v)\right)$ and using \cref{thm Jch} for $k+1$, we get 
$$\sum_{\nu}\left(\exp\left(\Binf(-t,\bfp,-\alpha v)\right)\cdot t^{|\nu|} p_\nu\right) \Jch_\nu(u_1,\dots,u_k)= \sum_{\mu}t^{|\mu|}\Jch_\mu(v,u_1,\dots,u_k) p_\mu.$$
Taking the limit over $k$ and extracting the coefficient of $p_\mu$, we get
$$[p_\mu]\sum_{\nu}\left(\exp\left(\Binf(-t,\bfp,-\alpha v)\right)\cdot t^{|\nu|}p_\nu\right) \Jch_\nu(u_1,u_2,\dots)=t^{|\mu|}\Jch_\mu(v,u_1,u_2,\dots).$$
We obtain \cref{eq 1 def skew Jch} by extracting the coefficient of $\Jch_\nu(u_1,u_2\dots)$. This equation can be rewritten as follows
\begin{align*}
  t^{|\mu|-|\nu|}\Jch_{\mu/\nu}(v)
  &=\left\langle\exp\left(\B_\infty(-t,\bfp,-\alpha v)\right)\cdot p_\nu,\frac{p_\mu}{z_\mu \alpha^{\ell(\mu)}}\right\rangle\\
  &=\left\langle p_\nu,\exp\left({\Binf}^\perp(-t,\bfp,-\alpha v)\right)\cdot\frac{p_\mu}{z_\mu \alpha^{\ell(\mu)}}\right\rangle\\
  &=[p_\nu]\exp\left({\B_\infty}^\perp(-t,\bfp,-\alpha v)\right)\cdot\frac{z_\nu \alpha^{\ell(\nu)}}{z_\mu \alpha^{\ell(\mu)}}p_\mu.\qedhere  
\end{align*}
\end{proof}
\end{prop}

\subsection{Proof of the main theorem}\label{ssec proof of the main thm}

\begin{proof}[Proof of \cref{thm diff eq}]
We have from \cref{eq 1 def skew Jch}
\begin{multline*}
  \exp\left(\B_\infty(-t,\bfq,-\alpha v)\right) \exp\left(\B_\infty(-t,\bfr,-\alpha v)\right)\cdot G^{(\alpha)}(t,\bfp,\bfq,\bfr)\\
  =\sum_{\pi,\mu,\nu}\left(\sum_{\rho,\xi}\frac{g^\pi_{\rho,\xi}(\alpha)}{z_\pi \alpha^{\ell(\pi)}}\Jch_{\mu/\rho}(v)\Jch_{\nu/\xi}(v)\right)t^{|\mu|+|\nu|-|\pi|}p_\pi q_\mu r_\nu.  
\end{multline*}
On the other hand, \cref{eq 2 def skew Jch} gives
$$\exp\left({\Binf}^\perp(-t,\bfp,-\alpha v)\right)\cdot G^{(\alpha)}(t,\bfp,\bfq,\bfr)=\sum_{\pi,\mu,\nu}\left(\sum_{\kappa}\frac{g^\kappa_{\mu,\nu}}{z_\kappa \alpha^{\ell(\kappa)}} \Jch_{\kappa/\pi}(v)\right) t^{|\mu|+|\nu|-|\pi|}p_\pi q_\mu r_\nu.$$
Combining these two equations with \cref{lem g skew Jch}, we deduce that
\begin{multline}
  \exp\big(\Binf(-t,\bfq,-\alpha v)+\Binf(-t,\bfr,-\alpha v)\big)\cdot G^{(\alpha)}(t,\bfp,\bfq,\bfr)
  \\=\exp\left({\Binf}^\perp(-t,\bfp,-\alpha v)\right)\cdot G^{(\alpha)}(t,\bfp,\bfq,\bfr).\label{eq exp on G}
\end{multline}
Since ${\Binf}^{\perp}(-t,\bfp,-\alpha v)$ commutes with each one of the operators $\Binf(-t,\bfq,-\alpha v)$ and $\Binf(-t,\bfr,-\alpha v)$ as operators in $\mathcal{O}(\mcA)\llbracket v\rrbracket$, we obtain by induction that for any $\ell\geq 1$ 
\begin{multline*}
  \exp\big(\Binf(-t,\bfq,-\alpha v)+\Binf(-t,\bfr,-\alpha v)\big)^\ell\cdot G^{(\alpha)}(t,\bfp,\bfq,\bfr)
  \\=\exp\left({\Binf}^\perp(-t,\bfp,-\alpha v)\right)^\ell\cdot G^{(\alpha)}(t,\bfp,\bfq,\bfr).
\end{multline*}
This allows to take the logarithm of the operators in \cref{eq exp on G}. We get that
$$\left(\Binf(-t,\bfq,-\alpha v)+\Binf(-t,\bfr,-\alpha v)\right)\cdot G^{(\alpha)}(t,\bfp,\bfq,\bfr)
={\Binf}^\perp(-t,\bfp,-\alpha v)\cdot G^{(\alpha)}(t,\bfp,\bfq,\bfr).$$
We conclude by replacing $v$ by $-u/\alpha$.
\end{proof}

Let us now prove \cref{thm com C-G}.
\begin{proof}[Proof of \cref{thm com C-G}]
    Using \cref{eq 1 def skew Jch},  \cref{lem g skew Jch} can be rewritten as the following commutation relation 
    \begin{multline*}
    \exp\big(\Binf(-t,\bfq,-\alpha v)+\Binf(-t,\bfr,-\alpha v)\big)\cdot \G(t,\bfp,\bfq,\bfr)
    \\=\G(t,\bfp,\bfq,\bfr)\cdot \exp\left(\Binf(-t,\bfp,-\alpha v)\right).    
    \end{multline*}
   Finally, we  "take the logarithm" of operators  as in the proof of \cref{thm diff eq}. 
\end{proof}

\subsection{Generalization to constellations}
The purpose of this subsection is to briefly explain how to generalize the differential equation of the main theorem to series with finitely many alphabets, and how these are related to constellations.

Fix an integer $k\geq 1$. 
We consider $k+2$ alphabets $\bfp$, $\bfq^{(0)}$, $\bfq^{(1)}$, \dots $\bfq^{(k)}$
and we introduce the following generalization of the series $G^{(\alpha)}$;
$$G^{(\alpha)}_k(t,\bfp,\bfq^{(0)},\dots,\bfq^{(k)}):=\sum_{\pi,\mu^{(0)},\dots ,\mu^{(k)}}\frac{g^{\pi}_{\mu^{(0)},\dots ,\mu^{(k)}}(\alpha)}{z_{\pi}\alpha^{\ell(\pi)}}t^{|\mu^{(0)}|+\dots |\mu^{(k)}|-|\pi|}p_\pi q^{(0)}_{\mu^{(0)}}\dots q^{(k)}_{\mu^{(k)}},$$
where $g^{\pi}_{\mu^{(0)},\dots ,\mu^{(k)}}$ are defined as the structure coefficients;
$$\Jch_{\mu^{(0)}}\dots \Jch_{\mu^{(k)}}=\sum_{\pi}g^\pi_{\mu^{(0)},\dots,{\mu^{(k)}}}(\alpha)\Jch_\pi.  
$$

We emphasize that the series $G^{(\alpha)}_k$ (or
more precisely the closely related series $\tau^{(\alpha)}_k$ below) gives access to
all Hurwitz numbers (and their $b$-deformation in the sense of \cite{ChapuyDolega2022})
with control on their full ramification profile. 

Using the same arguments as in the case of three alphabets, we obtain the following theorem.
\begin{thm}\label{thm constellations}
For any $k\geq1$, we have,
\begin{multline*}\label{eq constellations}
\left(\Binf(-t,\bfq^{(0)},u)+\dots+\Binf(-t,\bfq^{(k)},u)\right)\cdot G^{(\alpha)}_k(t,\bfp,\bfq^{(0)},\dots,\bfq^{(k)})
\\= {\Binf}^{\perp}(-t,\bfp,u) \cdot G^{(\alpha)}_k(t,\bfp,\bfq^{(0)},\dots,\bfq^{(k)}).  
\end{multline*}
\end{thm}

As mentioned in the introduction, such an equation
seems to be new even in the classical case $b=0$. We now explain the connection of $G^{(\alpha)}_k$ to the series of $k$-constellations. As in \cref{ssec MJ}, we consider the series 
$$\tau^{(\alpha)}_k(t,\bfp,\bfq^{(0)},\dots,\bfq^{(k)}):=\sum_{n\geq0}t^n\sum_{\theta\vdash n}\frac{1}{j^{(\alpha)}_\theta}J^{(\alpha)}_\theta(\mathbf{p})J^{(\alpha)}_\theta(\mathbf{q}^{(0)})\dots J^{(\alpha)}_\theta(\bfq^{(k)}).$$ 

It turns out that $\tau^{(\alpha)}_k$ corresponds to the series of orientable $k$-constellations when $\alpha=1$ \cite{JacksonVisentin1990}, and to the series of all $k$-constellations (orientable or not) introduced in \cite{ChapuyDolega2022} when $\alpha=2$; see \cite{BenDali2022}.
On the other hand, extending the proof of \cref{prop_c-g_2} to $k+2$ alphabets, we get 
$$\exp\left(\frac{p_1}{t\alpha}\right)G^{(\alpha)}_k(t,\bfp,\bfq^{(0)},\dots,\bfq^{(k)})
=\exp\left(\frac{\partial}{t\partial q^{(0)}_1}+\dots+\frac{\partial}{t
  \partial q^{(k)}_1}\right)\tau^{(\alpha)}_k(t,\bfp,\bfq^{(0)},\dots,\bfq^{(k)}).$$


\subsection{Operators \texorpdfstring{$\C_\ell$}{C}}\label{ssec C l}
For $\ell \geq0$ we consider the operator $\C_{\ell}(t,\bfp)$ given by 
$$\C_\ell(t,\bfp):=[u^\ell]\Binf(t,\bfp,u): \mcP\rightarrow\mcP\llbracket t\rrbracket.$$

The differential equation \cref{eq diff eq} of the main theorem  is then equivalent to the equations 
\begin{equation*}\label{eq diff eq l}
  \left(\C_\ell(-t,\bfq)+\C_\ell(-t,\bfr)\right)\cdot G^{(\alpha)}(t,\bfp,\bfq,\bfr)= {\C_\ell}^{\perp}(-t,\bfp) \cdot G^{(\alpha)}(t,\bfp,\bfq,\bfr), \quad \text{\hspace{-0.2cm}for } \ell \geq 0.  
\end{equation*}
Similarly, \cref{thm com C-G} is equivalent to 
\begin{equation}\label{eq com C-G}
    \left(\C_\ell(-t,\bfq)+\C_\ell(-t,\bfr)\right)\cdot\G(t,\bfp,\bfq,\bfr)=\G(t,\bfp,\bfq,\bfr)\cdot \C_\ell(-t,\bfp) \quad \text{\hspace{-0.2cm}for } \ell \geq 0.
    \end{equation}
    
We deduce the following corollary which will be useful in the solution of the differential equation of the main theorem in \cref{ssec explicit formula}.
\begin{cor}\label{cor com C-G}
    For any partition $\lambda=[\lambda_1,\dots,\lambda_s]$, we have
    \begin{equation}\label{eq com C-G 2}
    \prod_{1\leq i\leq s}\left(\C_{\lambda_i}(-t,\bfq)+\C_{\lambda_i}(-t,\bfr)\right)\cdot \G(t,\bfp,\bfq,\bfr)=\G(t,\bfp,\bfq,\bfr)\cdot \prod_{1\leq i\leq s}\C_{\lambda_i}(-t,\bfp).    
\end{equation}
\end{cor}

Actually, the product in \cref{eq com C-G 2} can be taken in any order, since the operators $\C_\ell$ satisfy the following commutation relations; see \cite[Theorem~6.6]{BenDaliDolega2023}.
\begin{thm}\cite{BenDaliDolega2023}\label{thm com C}
Let $m>0$. Then 
$$\left[\C_\ell,\C_m\right]=\left\{\begin{array}{cc}
     0 & \text{ if } \ell>0,  \\
    (m+1)\C_{m+1} & \text{ if }  \ell=0.
\end{array}\right.$$
\end{thm}

\section{Combinatorial proof of the differential equation for \texorpdfstring{$\alpha=1$}{alpha=1}}\label{sec alpha 1}
The purpose of this section is to give a combinatorial proof of the commutation relation \cref{eq com C-G} (equivalently the differential equation \cref{eq diff eq}) for $\alpha=1$. To this purpose, we start by recalling the combinatorial interpretation of the operators $\Cone_\ell$ given in \cite{BenDaliDolega2023}, see \cref{prop C op}. We then use \cref{prop H-G} to obtain a combinatorial interpretation of the operator $\Gone$, see \cref{cor G op}.

We believe that the combinatorial constructions of Subsections \ref{subs pre-hypermpas} and \ref{ssec op G} are of independent interest and might be useful to shed some light on the combinatorics of hypermaps with controlled profile.

All maps considered in this section are orientable. 


\subsection{Interpretation of the operator \texorpdfstring{$\C_\ell$}{C} for \texorpdfstring{$\alpha=1$}{alpha=1}}\label{ssec C alpha 1}
It will be more convenient at some steps of the proof to work with maps with labelled edges rather than labelled vertices.
\begin{defi}\label{def labelled maps}
We say that a map $M$ is \textit{labelled} if  its edges are numbered $1,2,\dots, |M|$. We say that a map is \textit{bipartite} if its vertices are colored in white and black and each edge connects two vertices of different colors. If $M$ is a bipartite map of size $n$, then its \textit{face-type} is the partition of $n$ obtained by reordering the face degrees divided by 2.
\end{defi}

The following proposition is a special case of \cite[Proposition 4.5]{BenDaliDolega2023}. Since we consider here a different convention of map labelling, we briefly explain the main ideas of the proof.

\begin{prop}[{\cite[Proposition 4.5]{BenDaliDolega2023}}]\label{prop C op}
    Fix an integer $\ell\geq 0$, a partition $\pi$, and let $N$ be a labelled orientable bipartite map of face-type $\pi$. Then, 
    \begin{equation}\label{eq_prop_Cop}
        \ell!\Cone_\ell(-t,\bfp)\cdot \frac{p_\pi}{|\pi|!}=\sum_{n\geq \ell}(-t)^n\sum_{M}\frac{p_{\textup{face-type}(M)}}{|M|!},
    \end{equation}
    where the second sum is taken over labelled orientable maps $M$ obtained from $N$ as follows:
    \begin{itemize}
        \item we add a black vertex $v$ and $\ell$ new white vertices $w_1$, $w_2$,\dots,$w_\ell$. 
        \item we add $n$ edges all incident to $v$ and such that each new white vertex $w_i$ is connected to $v$ by at least one edge (we do not put any restriction on the degrees of the new vertices $w_i$).
        \item we relabel the edges of $M$ in any way. 
    \end{itemize}
\end{prop}
An example of maps $N$ and $M$ is given in \cref{fig Cl action}.

\begin{figure}[t]
\centering
\begin{subfigure}{0.45\linewidth}
    \centering
    \includegraphics[width=0.8\textwidth]{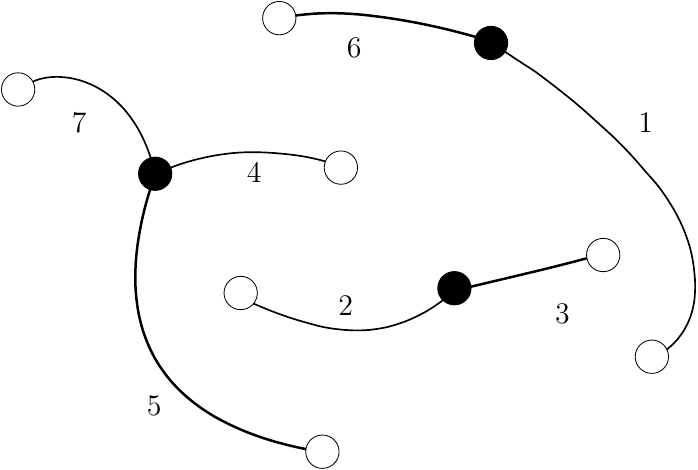}
    \begin{minipage}{2cm}
            \vfill 
    \end{minipage}
\end{subfigure}
\hfill
\begin{subfigure}{0.45\linewidth}
    \centering
    \includegraphics[width=0.8\textwidth]{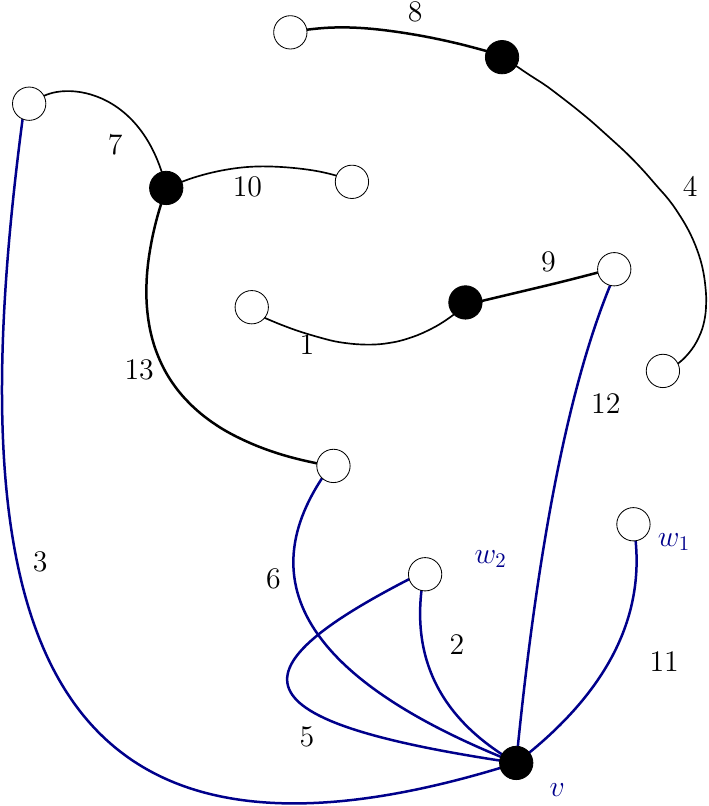}
\end{subfigure}
    \caption{Example of the action of $\Cone_2$ on a map $N$. On the left the map $N$, and on the right a map $M$ obtained by adding one black vertex $v$, two white vertices $w_1$ and $w_2$ and 6 edges (represented in blue).}
    \label{fig Cl action}
\end{figure}
\begin{proof}
    We recall that the operator $\C_\ell$ is defined by 
\begin{equation}\label{eq Bone-Cone}
  \Cone_\ell(-t,\bfp)=[u^\ell]\sum_{n\geq 1}(-t)^n\frac{\Bone_n(\bfp,u)}{n}.  
\end{equation}
where
\begin{equation*}
  \Bone_n(\bfp,u):=\Theta_Y\left(\GY^{(\alpha=1)}+u Y_+\right)^{n}y_0
\end{equation*}
see \cref{ssec op Bn} for the definitions of the operators $\Theta_Y$, $\GY$ and $Y_+$. 

In order to understand the combinatorics of these operators, we consider \textit{rooted maps}. A map is rooted if it has a marked corner $c$ called the root corner. We associate to a rooted map the weight
$$\weight(M,c):=y_{\deg(f_c)}\prod_{f\neq f_c} p_{\deg(f)}\in\PY,$$
where $\deg(f_c)$ denotes the degree of the root face, and the  the product runs over all faces of $M$ different from $f_c$. When $\alpha=1$, the operator $\Bone$ can be interpreted as follows:
$$\Bone_n(\bfp,u)\cdot p_{\textup{face-type}(N)} u^{|\Vcirc(N)|}=\sum_{M} p_{\textup{face-type}(M)}u^{|\Vcirc(M)|},$$
where $|\Vcirc(.)|$ denotes the number of white vertices and where the sum is taken over orientable maps $M$ obtained from $N$ by adding a black vertex $v$ of degree $n$ with a rooted corner.

Let us briefly explain this equation.
First, we add an isolated root vertex, this corresponds to multiplication by $y_0$. We then add consecutively $n$ edges to the root corner: the added edge can be connected to a new white vertex (this corresponds to the term $u Y_+$) or to an existing white vertex in the map (ensured by the operator $\GY$). Finally we apply $\Theta_Y$ to obtain the weight of an unrooted map; $\Theta_Y \cdot \weight(M,c)=p_{\textup{face-type(M)}}$. We refer to \cite[Section 4.2]{BenDaliDolega2023} for more details about the combinatorics of these operators. 

If $N$ and $M$ are now labelled maps then the previous equation becomes
$$\frac{\Bone_n(\bfp,u)}{n}\cdot \frac{p_{\tf(N)}}{|N|!} u^{|\Vcirc(N)|}=\sum_{M} \frac{p_{\tf(M)}}{|M|!}u^{|\Vcirc(M)|}.$$
Indeed, we have $|M|!$ ways to choose new labels for the edges of $M$, and then we divide by $|N|!$ to "forget" the old labels in $N$, and by $n$ to forget the root of the added black vertex $v$. 


Since in \cref{eq Bone-Cone} we take the coefficient of $[u^{\ell}]$, then the operator $\mathcal{C}_\ell(-t,\bfp)$ acts on a map by adding a black vertex using $\ell$ new white vertices with an extra weight $-t$ for each added edge. Finally, multiplying by $\ell!$ in \cref{eq_prop_Cop} corresponds to having a total order on the new added white vertices $w_1$, $w_2$, \dots $w_\ell$.
\end{proof}

\subsection{\texorpdfstring{\textbf{BFC}}{BFC} maps and pre-hypermaps}\label{subs pre-hypermpas}
We start by introducing a family of maps which allows to encode the hypermaps with marked faces defined in \cref{def hypermaps}.  

\begin{defi}\label{def FV-colored}
We say that a map is \emph{bipartite face-colored} $($\BFC map$)$ if its vertices are colored in black and white, its faces are colored in two colors $(+)$ and $(-)$, and such that each edge connects two vertices of different colors (but it does not necessarily separate two faces of different colors).

Moreover, a \BFC  map will be called a \textit{pre-hypermap} if it satisfies the following additional conditions:
\begin{enumerate}
    \item white vertices have degree at most 2.
    \item all white vertices of degree 2 are incident to two faces of different colors.
\end{enumerate}
\end{defi}

\begin{rmq}\label{rmq hyper pre hyper}
Notice that a hypermap can be seen as a pre-hypermap; we color all the vertices of the hypermap in black and add in the middle of each edge a white vertex of degree 2. Hence, hypermaps are pre-hypermaps maps with only white vertices of degree 2. Conversely, if we delete all white vertices of degree 1 in a pre-hypermap we obtain a hypermap.
\end{rmq}
We distinguish two types of edges in a \BFC map.
\begin{defi}
An edge is said \textit{bicolor} if it is incident to two faces of different colors. We have two types of bicolor edges in a \BFC map $($see \cref{fig:types_defi}$)$:
\begin{itemize}
    \item Type 1: on the $(+)$ side-face, we see the white vertex and then the black one when we travel along the edge-side in the direct orientation. 
    \item Type 2: on the $(+)$ side-face, we see the black vertex and then the white one when we travel along the edge-side in the direct orientation. 
\end{itemize}
By definition, all non bicolor edges will be considered of type 1.
\end{defi}
\begin{figure}
    \centering
    \includegraphics[width=0.4\textwidth]{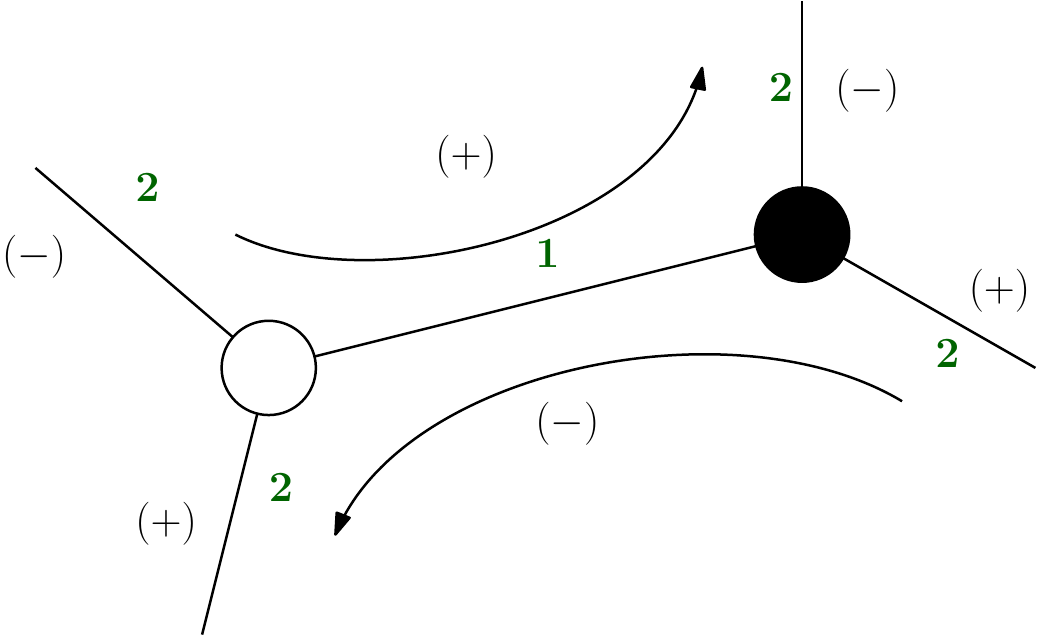}
    \caption{Types of bicolor edges in a \BFC map.}
    \label{fig:types_defi}
\end{figure}
\begin{rmq}\label{rmq bicolored and type 2}
    Note that the types can be equivalently defined by conditions around vertices. For instance, when we turn around a black vertex in the direct orientation, then edges of type 2 are exactly those who separate faces $(-)/(+)$ in this order. Consequently, a vertex which is not monochromatic (\textit{i.e.} incident at least to one $(+)$ and one $(-)$ face) is necessarily incident to a type 2 edge.
\end{rmq}
 It is easy to check that each black vertex of a pre-hypermap has even number of white neighbors of degree 2 (see \textit{e.g} \cref{rmq hyper pre hyper}). This allows to define the degree of a black vertex $v$ in a pre-hypermap as follows
\begin{align}\label{eq deg pre hypermap}
  \deg(v)&=\frac{1}{2}\lvert\left\{\text{neighbors of $v$ of degree 2}\right\}\rvert+\lvert\left\{\text{neighbors of $v$ of degree 1}\right\}\rvert\\
  &=\lvert\left\{\text{incident edges to $v$ of type 1}\right\}\rvert\label{eq deg pre hypermap 2}.  
\end{align}

We now extend the notion of profile to pre-hypermaps.
\begin{defi}\label{def profile pre hypermap}
Let $M$ be a pre-hypermap. We denote by $\tblack(M)$ the partition given by black vertices degrees $($as defined in \cref{eq deg pre hypermap}$)$. We also denote by $\tplus(M)$ $($resp. $\tminus(M)$ the partition given by the $(+)$ face $($resp. the $(-)$ face$)$ degrees divided by 2.
We call \textit{the profile} of $M$ the triplet of partitions $(\tblack(M),\tplus(M),\tminus(M))$.
\end{defi}

One can check that the profile of a hypermap as defined in \cref{def profile hypermap} coincides with its pre-hypermap profile as defined in \cref{def profile pre hypermap}.

\begin{defi}\label{def vertex labelled pre hypermaps}
We say that a pre-hypermap is \textit{vertex labelled} if:
\begin{itemize}
    \item for each $d\geq 1$, vertices of same degree $d$ are numbered 1,2,\dots. 
    \item each black vertex has a marked corner, oriented in the direct orientation and followed by an edge  of type 1.
\end{itemize}   
Fix three partitions $\pi$, $\mu$ and $\nu$. We define $\oprehypermap^\pi_{\mu,\nu}$ as the set of vertex-labelled oriented pre-hypermaps of profile $(\pi,\mu,\nu)$.
\end{defi}

The following lemma connects pre-hypermaps to hypermaps with marked faces.
\begin{lem}\label{lem bij ohyp-oprehyp}
    Fix three partitions $\pi$, $\mu$ and $\nu$. There is a bijection between $\ohypermap^\pi_{\mu,\nu}$ (defined in \cref{def hypermaps}) and $\oprehypermap^\pi_{\mu,\nu}$.
    \begin{proof}
    First, notice that each degree 2 face of a pre-hypermap (degree 1 face in the associated hypermap) is incident exactly to one edge of type 1 and one edge of type 2. Notice also that  the only case in which an edge is incident to two faces of degree 2 is the case of an isolated loop.

    As in \cref{rmq hyper pre hyper}, we can think of $M$ as a pre-hypermap which we denote $M'$. First, notice that each degree 2 face in $M'$ (degree 1 face in $M$) is incident exactly to one edge of type 1 and one edge of type 2. Notice also that  the only case in which an edge of type 2 is incident to two faces of degree 2 is the case of an isolated loop.
    
    By deleting in $M'$ all edges of type 2 incident to a marked face and forgetting the colors of these faces, we obtain a map $N$ in $\oprehypermap^\pi_{\mu,\nu}$, see \cref{fig hyper to prehyper}. Indeed, in this procedure the degree of a black vertex (as defined in \cref{eq deg pre hypermap}) is unchanged. Moreover, each one of the faces of $N$ inherits a color from $M'$ and its degree is unchanged. 

    Conversely, from $N\in\oprehypermap^\pi_{\mu,\nu}$ we obtain a map $M\in\ohypermap^\pi_{\mu,\nu}$ as follows; first we transform each white leaf into a loop, we mark the formed 2-degree face and we color it so that the added edge is bicolor. 
    \begin{figure}[t]
        \centering
        \begin{subfigure}{0.45\textwidth}
        \includegraphics[width=1\textwidth]{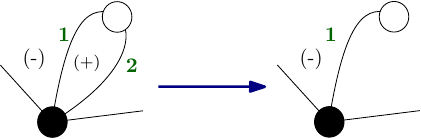} 
        \subcaption{The marked face is colored $(+)$}
        \end{subfigure}
        \begin{subfigure}{0.45\textwidth} 
        \includegraphics[width=1\textwidth]{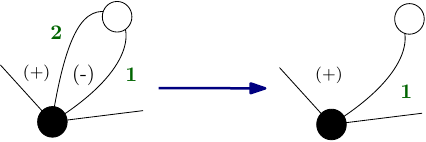}
        \subcaption{The marked face is colored $(-)$}
        \end{subfigure}
        \caption{Deleting edges of a hypermap with marked faces to obtain a pre-hypermap.}
        \label{fig hyper to prehyper}
    \end{figure}
    \end{proof}
\end{lem}
An example of the correspondence described above is given in \cref{fig map hyper to prehyper}.
\begin{figure}[t]
        \centering
        \begin{subfigure}{0.45\textwidth}
        \includegraphics[width=1\textwidth]{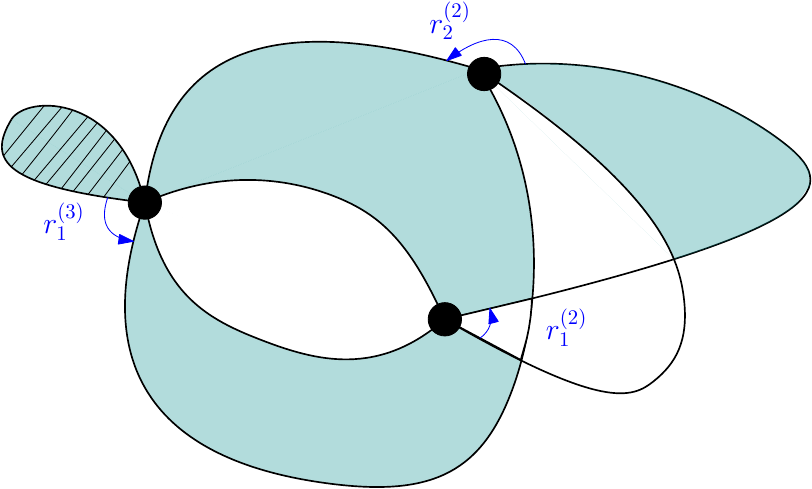}    
        \end{subfigure}
        \begin{subfigure}{0.45\textwidth}
        \includegraphics[width=0.85\textwidth]{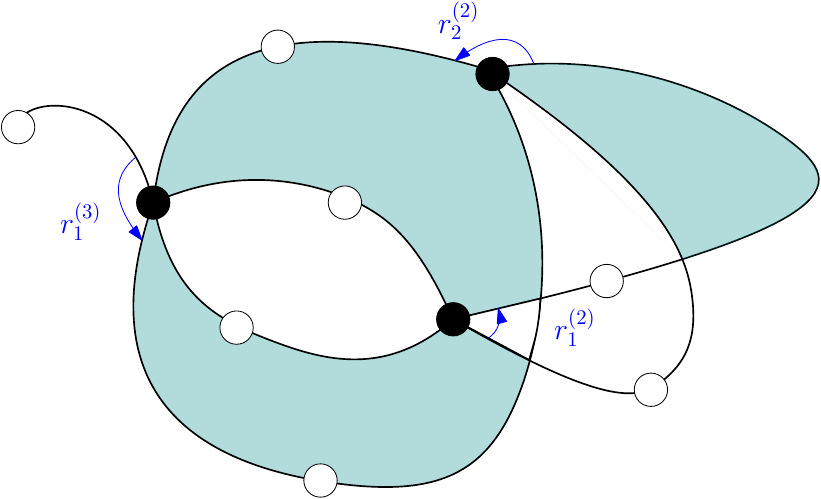}    
        \end{subfigure}
        \caption{On the left a labelled hypermap with one marked face; $(-)$ faces are represented in blue and the marked face is crossed. On the right the associated  pre-hypermap.}
        \label{fig map hyper to prehyper}
    \end{figure}

\begin{prop}\label{prop g0-pre hyper}
    For any partitions $\pi$,$\mu$ and $\nu$, we have
    $$g^\pi_{\mu,\nu}(1)=\lvert\oprehypermap^\pi_{\mu,\nu}\rvert.$$
    Equivalently, $(|\mu|+|\nu|)!/z_\pi g^\pi_{\mu,\nu}(1)$ is the number of labelled orientable pre-hypermaps of profile $(\pi,\mu,\nu)$ (see \cref{def labelled maps}). 
    \begin{proof}
        We know from \cref{lem bij ohyp-oprehyp} that $\lvert\oprehypermap^\pi_{\mu,\nu}\rvert=\lvert\ohypermap^\pi_{\mu,\nu}\rvert$.  On the other hand $g^\pi_{\mu,\nu}(1)=\lvert\ohypermap^\pi_{\mu,\nu}\rvert$ by \cref{eq prop H-G}, this gives the first part of the proposition.

        To obtain the second part, we start by noticing that a pre-hypermap of profile $(\pi,\mu,\nu)$ has $|\mu|+|\nu|$ edges. Moreover, to pass from a vertex-labelled hypermap of vertex type $\pi$ to a labelled hypermap, we start by labelling edges and then we forget vertex labels which corresponds to multiplying by $(|\mu|+|\nu|)!/z_\pi$.
    \end{proof}
\end{prop}

\subsection{Combinatorial interpretation of \texorpdfstring{$\Gone$}{G}}\label{ssec op G}

In this subsection, we give a second combinatorial interpretation for $g^\pi_{\mu,\nu}(1)$ which generalizes \cref{prop g0-pre hyper}. This interpretation consists in seeing the series of hypermaps as an operator (see \cref{cor G op}) rather than a "static" generating series.

Fix a partition $\pi$. We call $\pi$\textit{-star} map the unique bipartite map with only white vertices of degree 1 and black vertices of type $\pi$. Notice that labelled $\pi$-star maps are in bijection with permutations of cyclic type $\pi$. In particular, there are $|\pi|!/z_\pi$ such maps.

\begin{lem}\label{lem star to pre-hyper}
    Let $M$ be a pre-hypermap of profile $(\pi,\mu,\nu)$. Then the map obtained by deleting all edges of type 2 is the $\pi$-star map. 
    
    Conversely, let $M$ be a \BFC map  such that $\tplus(M)=\mu$ and $\tminus(M)=\nu$. Assume that $M$ is obtained by adding $|\mu|+|\nu|-|\pi|$ edges to the $\pi$-star map and by coloring the faces, such that the added edges are bicolor of type 2. Then all white vertices of $M$ have degree 1 or 2, white vertices of degree 2 are incident to two faces of different colors, and $\lambda^\bullet(M)=\pi$. In other terms, $M$ is a pre-hypermap of profile $(\pi,\mu,\nu)$.

    \begin{figure}[t]
        \centering
        \begin{subfigure}{0.45\textwidth}
        \includegraphics[width=1\textwidth]{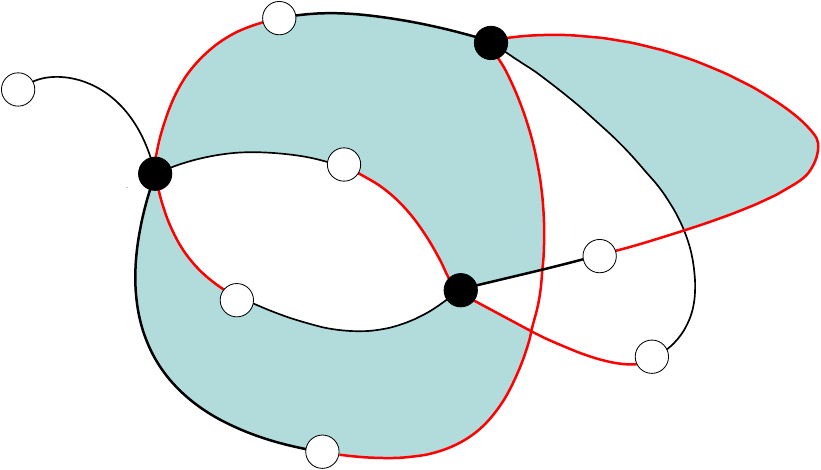}    
        \end{subfigure}
        \begin{subfigure}{0.45\textwidth}
        \includegraphics[width=0.85\textwidth]{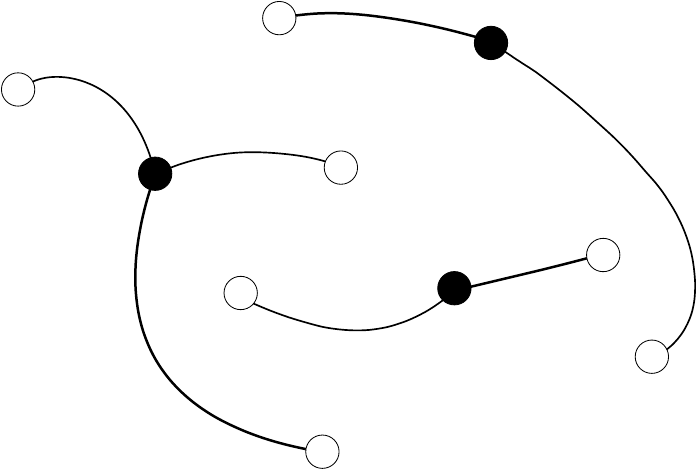}    
        \end{subfigure}
        \caption{On the left an oriented pre-hypermap, faces colored in $(-)$ are represented in blue. Bicolor edges of type 2 are represented in red. On the right the $[3,2,2]$-star map obtained by deleting edges of type 2.}
        \label{fig deleting type 2}
    \end{figure}
    \begin{proof}
        We start by proving the first assertion. Let $M$ be a pre-hypermap of profile $(\pi,\mu,\nu)$ and let $N$ be the map obtained by deleting all edges of type 2 (an example is given in \cref{fig deleting type 2}). By definition all white vertices of $M$ have degree 1 or 2. Moreover, each white vertex of degree two is incident to one edge of type 1 and one edge of type 2. Hence, $N$ is a star map. Furthermore, the degree of a black vertex in a pre-hypermap is given by the number of incident edges of type 1 (see \cref{eq deg pre hypermap 2}), and is then unchanged by deleting edges of type 2. By consequence the type of black vertices in $N$ is the same as in $M$, that is $\pi$. 

        Let us now prove the second assertion. Let $M$ be a \BFC map obtained from the $\pi$-star map as above. Since all added edges are of type 2, we can not add two edges incident to the same white corner. Hence, all white vertices in $M$ have degree 1 or 2. Moreover, white vertices of degree 2 are incident to two faces of different colors since the added edges are bicolor. This proves that $M$ is a pre-hypermap. The type of black vertices of $M$ is $\pi$ by the same argument as above. This finishes the proof of the lemma.
    \end{proof}
\end{lem}



We deduce a combinatorial interpretation of $g^\pi_{\mu,\nu}(1)$ which generalizes \cref{prop g0-pre hyper}. 
\begin{prop}\label{prop const hypermaps}
Fix three partitions $\pi$, $\mu$ and $\nu$, and set $m:=|\pi|$ and $n:=|\mu|+|\nu|-|\pi|$.
Let $N$ be a labelled orientable bipartite map of face-type $\pi$. 
Then $\frac{n!}{m!}g^\pi_{\mu,\nu}(1)$ is the number of ways of obtaining a labelled \BFC map $M$ from $N$ by:
\begin{enumerate}
    \item adding $|\mu|+|\nu|-|\pi|$ edges to $N$  to obtain a map of face type $\mu\cup\nu$,
    \item coloring the faces such that the added edges are bicolor of type 2, and such that the obtained \BFC map $M$ satisfies $\lambda^+(M)=\mu$ and $\lambda^-(M)=\nu$,
    \item relabelling all the edges.
\end{enumerate}
\end{prop}
Note that the number of ways of obtaining a \BFC $M$ by the three operations described above from a bipartite map $N$ depends only on the face-type of $N$ but does not depend on its structure as will be shown in the proof.
\begin{proof}
We start by proving the result when $N$ is a labelled $\pi$-star map. 
We know from \cref{prop g0-pre hyper} that the number of labelled pre-hypermaps of profile $(\pi,\mu,\nu)$ is $n!/z_\pi g^\pi_{\mu,\nu}(1)$. We now count in a different way the number of labelled pre-hypermaps of profile $(\pi,\mu,\nu)$.

Let $f^\pi_{\mu,\nu}$ be the number of ways of realizing the steps $(1)$, $(2)$ and $(3)$ described above starting from a fixed labelled star map of face-type $\pi$.

In order to obtain a pre-hypermap of profile $(\pi,\mu,\nu)$, we start by choosing a labelled $\pi$-star map (we have $m!/z_\pi$ choices), we then have $f^\pi_{\mu,\nu}$ ways to add edges to obtain a labelled pre-hypermap of profile  $(\pi,\mu,\nu)$  (we use here \cref{lem star to pre-hyper}).  Hence 
$$n!/z_\pi g^\pi_{\mu,\nu}(1)=m!/z_\pi f^\pi_{\mu,\nu}.$$
We deduce that $f^\pi_{\mu,\nu}=\frac{n!}{m!}g^\pi_{\mu,\nu}(1)$. This finishes the proof for star maps.

In order to obtain the assertion for any bipartite labelled map $N$ of face-type $\pi$, we prove that the number of ways to realize the steps  \textit{(1), (2)} and \textit{(3)} on a map $N$ depends only on the face-type of the map and not on its structure. Indeed, when we have two maps of the same face-type, one can always find a bijection between the corners of the two maps which preserves the face structures; two corners are consecutive when we travel along a face (in the direct orientation) in the first map, if and only if their images in the second map satisfy the same property. Once such bijection is fixed, each way of adding edges and coloring faces on one map can be copied on the second map in a unique way which respects the bijection of the corners. The two maps obtained have necessarily the same $(+)$ and $(-)$ types (but not necessarily the same vertex degrees). We give an example in~\cref{fig faces isomorphism}. 
   \end{proof}
    

\begin{figure}[t]
\centering
\begin{subfigure}{0.45\linewidth}
    \centering
    \includegraphics[width=0.8\textwidth]{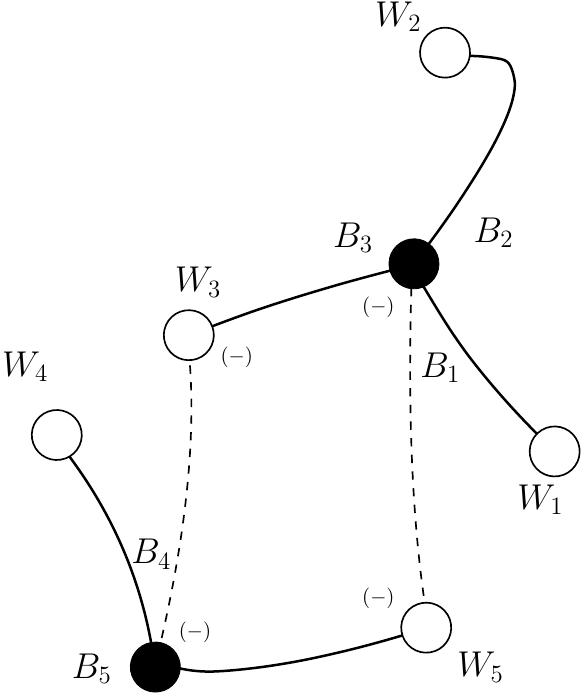}
    \begin{minipage}{2cm}
            \vfill 
    \end{minipage}
\end{subfigure}
\hfill
\begin{subfigure}{0.45\linewidth}
    \centering
    \includegraphics[width=0.8\textwidth]{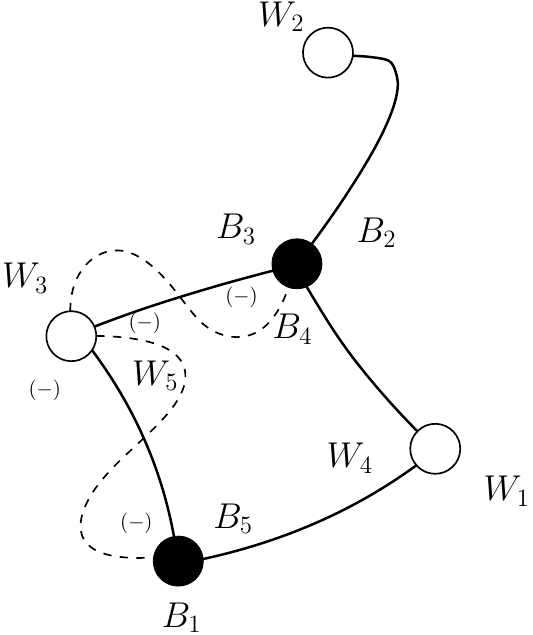}
\end{subfigure}
    \caption{The plain edges represent two initial maps with the same face-type. The labels $W_i$ and $B_i$ give a bijection between the corners of the two maps which preserves the face structure. The two dashed edges are added and faces colored with respect to this bijection; an edge between corners $(B_4,W_3)$ and an edge between $(B_1,W_5)$. Corners incident to a face of color $(-)$ are marked with a sign $(-).$}
    \label{fig faces isomorphism}
\end{figure}


Given \cref{prop const hypermaps}, it is possible to think of pre-hypermaps as partially constructed hypermaps. Indeed, hypermaps are obtained by adding a maximal number of edges on $\pi$-star maps. 
We now deduce a combinatorial interpretation of $\Gone$.
\begin{cor}\label{cor G op}
Fix a partition $\pi$ of size $m\geq 0$.
    Let $N$ be a labelled orientable bipartite map of face-type $\pi$. Then, 
$$\Gone(t,\bfp,\bfq,\bfr)\cdot \frac{p_{\pi}}{m!}=\sum_{M}t^{|M|-|N|}\frac{q_{\lambda^{+}(M)}r_{\lambda^{-}(M)}}{|M|!}$$
where the sum is taken over all ways to add edges to $N$ and to color faces, in order to obtain a \BFC orientable map $M$ such that the added edges are of type 2.
\end{cor}

\subsection{End of the combinatorial proof}
Through this subsection, we fix once and for all an integer $\ell \geq 0$, a partition $\pi$  and a labelled bipartite map $N$ of face-type $\pi$. Our goal is to use \cref{prop C op} and \cref{cor G op} to prove that, for $\alpha=1$, both sides of \cref{eq com C-G} act in the same way on the weight of $N$ given by $\frac{p_\pi}{|\pi|!}=\frac{p_{face-type(N)}}{|N|!}$. This would imply the commutation relation of \cref{eq com C-G} since power-sum functions are a basis of $\mcS$.  

The following definition will be useful in the combinatorial proof of \cref{eq com C-G}; all along this subsection, $\mathcal{M}$ will be the (infinite) set of labelled \BFC maps which are be obtained from $N$ by

\begin{itemize}
    \item adding one black vertex $v$ and $\ell$ white vertex $w_1,w_2,\dots, w_\ell$,
    \item adding some edges, such that each one of the new vertices is incident at least to one edge,
    \item choosing a color for each face,
    \item relabelling edges.
\end{itemize}
Moreover, in such a map we the added edges are marked.

Fix a \BFC map $M$ in $\mathcal{M}$. We denote by $\mathcal{E}_v(M)$ the set of edges incident to $v$ in $M$, and we denote $\mathcal{T}_2(M\backslash N)$ the set of edges of type 2 in $M$ not contained in $N$.

\medskip

On one hand, we have
\begin{multline}
    \label{eq C-G on map}\ell!\left(\Cone_\ell(-t,\bfq)+\Cone_\ell(-t,\bfr)\right)\cdot \Gone(t,\bfp,\bfq,\bfr)\cdot \frac{p_\pi}{|N|!}\\
  =\sum_{n\geq \ell}(-1)^n\sum_{M} t^{|M|-|N|}\frac{q_{\tplus(M)}r_{\tminus(M)}}{|M|!}, 
\end{multline}
where the second sum is taken over \BFC maps in $\mathcal{M}$ obtained from $N$ as follows
\begin{itemize}
    \item \textbf{Step 1:} We add edges to $N$ and we color the faces of the obtained map, such that the added edges are of type 2.
     \item \textbf{Step 2:} We start by choosing either the $(+)$ or the $(-)$ part of the map, and we add a black vertex $v$ of degree $n$ and $\ell$ white vertices only connected to $v$, such that all added edges are incident to faces of the chosen color.
\end{itemize}
After each one of these operations we relabel all the edges.

Note that in this construction we can not obtain a map $M\in \mathcal{M}$ in two different ways, since all edges added in \textbf{Step 1} are bicolor while those added in \textbf{Step 2} are not. More precisely, the right-hand side of \cref{eq C-G on map} can be rewritten as follows
$$\sum_{M\in\mathcal{M}^{(1)}}(-1)^{\lvert\mathcal{E}_v(M)\rvert} \frac{q_{\tplus(M)}r_{\tminus(M)}}{|M|!},$$
where 
\begin{align*}
\mathcal{M}^{(1)}:=&\left\{M\in\mathcal{M} \text{ such that  an added white vertex $w_i$ is only connected to $v$}\right.\\
&\left.\text{ and such that }\mathcal{E}(M\backslash N)=\mathcal{T}_2(M\backslash N)\uplus \mathcal{E}_v(M), \text{ the union being disjoint.} \right\}    
\end{align*}
Actually, the following lemma allows us to simplify the definition of $\mathcal{M}^{(1)}$.  

\begin{lem}
    Fix a \BFC map $M\in\mathcal{M}$. If  
    \begin{equation}\label{eq sufficient condition}
      \mathcal{E}(M\backslash N)=\mathcal{T}_2(M\backslash N)\uplus \mathcal{E}_v(M)  
    \end{equation}
    then the new white vertices $w_i$ are all only connected to $v$.
    \begin{proof}
        Let us suppose that there exists an added white vertex $w_i$ which is incident to a black vertex different from $v$. Since added edges are either incident to $v$ or of type 2 (see \cref{eq sufficient condition}), this implies that $w_i$ is incident to a bicolor edge $e$, and by consequence is incident to faces of different colors. As $e$ is of type 2, $w_i$ can not have degree 1. Let $e_1$ and $e_2$ denote the edges forming a corner in $w_i$ with $e$ ($e_1$ and $e_2$ are not necessarily distinct). But since we can not have two consecutive edges of type 2 around a vertex, then $e_1$ and $e_2$ are both incident to $v$, and are besides  incident to faces of different colors. We deduce that  $v$ is also incident to faces of different colors. By \cref{rmq bicolored and type 2} we get that $v$ is incident to an edge of type 2, this contradicts the fact that the union is disjoint in  \cref{eq sufficient condition}.
    \end{proof}
\end{lem}
We deduce that
$$\mathcal{M}^{(1)}:=\left\{M\in\mathcal{M} \text{ such that }
\mathcal{E}(M\backslash N)=\mathcal{T}_2(M\backslash N)\uplus \mathcal{E}_v(M)\right\}.$$

On the other hand, we have
\begin{equation}\label{eq G-C on map}
 \ell! \Gone(t,\bfp,\bfq,\bfr)\cdot \Cone_\ell(-t,\bfp)\cdot \frac{p_\pi}{|N|!}=\sum_{n\geq \ell}(-1)^n \sum_{M} t^{|M|-|N|}\frac{q_{\tplus(M)}r_{\tminus(M)}}{|M|!}  
\end{equation}
where the second sum runs over the set of labelled \BFC maps obtained from $N$ as follows;
\begin{itemize}
    \item \textbf{Step 1'}: we add a black vertex $v$ of degree $n$ and $\ell$ white vertices only connected to $v$.
    \item \textbf{Step 2'}: we add edges and color faces such that the edges added in this step are bicolor and of type 2.
\end{itemize}
Note that all these maps are in the set
\begin{align*}
\mathcal{M}^{(2)}:=&\left\{M\in\mathcal{M} \text{ such that $\mathcal{E}(M\backslash N)=\mathcal{E}_v(M)\cup\mathcal{T}_2(M\backslash N)$,}\right.\\
&\hspace{6cm}\left.\text{ the union not necessarily disjoint}\right\}.   
\end{align*}
It is straightforward that $\mathcal{M}^{(1)}\subseteq \mathcal{M}^{(2)}$. Our goal is to prove that the total contribution of maps in  $\mathcal{M}^{(2)}\backslash \mathcal{M}^{(1)}$ in \cref{eq G-C on map} is 0.  Indeed, any map in $M\in\mathcal{M}^{(2)}$ contributes in \cref{eq G-C on map} with a coefficient
\begin{equation}\label{eq contribution M}
  \sum_{\mathcal{E}(M\backslash N)=\mathcal{I}^{(1)}\uplus \mathcal{I}^{(2)}}(-1)^{\lvert\mathcal{I}^{(1)}\rvert}\frac{q_{\tplus(M)}r_{\tminus(M)}}{|M|!},  
\end{equation}
where the sum runs over all possible ways to decompose  the set edges of $M\backslash N$ into $\mathcal{I}^{(1)}$ and $\mathcal{I}^{(2)}$
such that $\mathcal{I}^{(1)}\subseteq \mathcal{E}_v(M)$ and $\mathcal{I}^{(2)}\subseteq \mathcal{T}_2(M)$. The only edges for which we have a choice (they can be either in $\mathcal{I}^{(1)}$ or in $\mathcal{I}^{(2)}$) are exactly the edges in $\mathcal{T}_2(M\backslash N)\cap \mathcal{E}_v(M)$. Let $r(M):=\lvert\mathcal{T}_2(M\backslash N)\cap \mathcal{E}_v(M)\rvert$. Then, \cref{eq contribution M} can be rewritten as follows
$$\sum_{i=0}^{r(M)} (-1)^{i+(|M|-|N|-\mathcal{T}_2(M\backslash N))}\binom{r(M)}{i}\frac{q_{\tplus(M)}r_{\tminus(M)}}{|M|!}=\left\{
\begin{array}{cc}
     0& \text{if $r(M)>0$ } \\
     (-1)^{\lvert\mathcal{E}_v(M)\rvert} \frac{q_{\tplus(M)}r_{\tminus(M)}}{|M|!}& \text{ if $r(M)=0$}.
\end{array}\right.$$
This finishes the combinatorial proof of \cref{thm com C-G} for $\alpha=1$.

\section{Solution of the differential equation}\label{sec resolution}
The main purpose of this section is to solve the differential equation of the main theorem to give an explicit expression of the structure coefficients $g^\pi_{\mu,\nu}(\alpha)$, see \cref{thm g-a}. As a byproduct of this result we construct an algebra isomorphism between the space of symmetric functions and space of shifted symmetric functions (\cref{cor iso}). Finally, we prove \cref{thm low terms} in \cref{ssec low terms}.
\subsection{Explicit expression of coefficients \texorpdfstring{$g^\pi_{\mu,\nu}$}{g}}\label{ssec explicit formula}
We define the coefficients $a^\lambda_\xi$ for any partitions $\lambda$ and $\xi$ by 
\begin{align*}
  a^\lambda_\xi:&=[t^{|\xi|}p_\xi]\left(\prod_{i=1}^{\ell(\lambda)}\alpha\lambda_i \C_{\lambda_i}(t,\bfp)\right)\cdot 1.
\end{align*}
Note that by \cref{thm com C}, the product in the last equation can be taken in any order.
Using the combinatorial interpretation of the operators $\C_\ell$ given in \cite{BenDaliDolega2023} (see \cref{ssec C alpha 1} for the case $\alpha=1$), it is possible to give a combinatorial interpretation for the coefficients $a^\lambda_{\xi}$ in terms of layered maps introduced in \cite{BenDaliDolega2023}.


We also consider the coefficients $d^\lambda_{\mu,\nu}$ defined by
$$d^\lambda_{\mu,\nu}:=\sum_{\xi\cup\pi=\lambda}a^\xi_\mu a^\pi_\nu=[t^{|\mu|+|\nu|}q_\mu r_\nu]\prod_{1\leq i\leq \ell(\mu)}\alpha \lambda_i\left(\C_{\lambda_i}(t,\bfq) +\C_{\lambda_i}(t,\bfr)\right)\cdot 1,$$
where the sum is taken over all possible ways of grouping the parts of $\lambda$ into two partitions $\xi$ and $\pi$.

It follows from the definition of operators $\C_\ell$ that 
\begin{equation}
  [t^k]\C_\ell(t,\bfp)=\left\{
\begin{array}{cc}
   0  & \text{if }k<\ell \\
 p_\ell/(\alpha\ell)& \text{if } k=\ell,
\end{array}
\right.  
\end{equation}
see also \cref{eq C_0} for details. We deduce that,
\begin{equation}\label{eq top a coef}
  a^\lambda_\xi=\left\{
\begin{array}{cc}
   0  & \text{if }|\xi|<|\lambda| \\
    \delta_{\lambda,\xi} & \text{if } |\xi|=|\lambda|.
\end{array}
\right.  
\end{equation}

\noindent Similarly, $d^\lambda_{\mu,\nu}=0$ if $|\mu|+|\nu|-|\lambda|<0$.

\noindent We now state the main result of this section.

\begin{thm}\label{thm g-a}
For any partitions $\lambda,\mu$ and $\nu$ we have
\begin{equation}\label{eq g-a}
  g^\lambda_{\mu,\nu}=(-1)^{|\mu|+|\nu|-|\lambda|}\sum_{m\geq 0}(-1)^m\sum_{\pi_1,\dots,\pi_m\atop{|\lambda|<|\pi_1|<\dots<|\pi_m|}} a^\lambda_{\pi_1}a^{\pi_1}_{\pi_2}a^{\pi_2}_{\pi_3}\dots a^{\pi_{m-1}}_{\pi_m}d^{\pi_m}_{\mu,\nu},  
\end{equation}
The term $m=0$ in the second sum is interpreted as $d^\lambda_{\mu,\nu}$.
\begin{proof}
We fix $\mu$ and $\nu$, and we proceed by induction on $|\mu|+|\nu|-|\lambda|$. If $|\mu|+|\nu|-|\lambda|<0$ then the equality holds since $g^\lambda_{\mu,\nu}=0$ (see \cref{lem deg bound}).
Using \cref{cor com C-G}, we write
$$\prod_{1\leq i\leq s}\left(\C_{\lambda_i}(-t,\bfq)+\C_{\lambda_i}(-t,\bfr)\right)\cdot \G(t,\bfp,\bfq,\bfr) \cdot 1=\G(t,\bfp,\bfq,\bfr)\prod_{1\leq i\leq s}\C_{\lambda_i}(-t,\bfp)\cdot 1.$$
But from \cref{rmq empty g} and \cref{lem deg bound}, we know that 
\begin{equation}\label{eq G.1}
\G(t,\bfp,\bfq,\bfr) \cdot 1=1.
\end{equation}
Hence,
$$\prod_{1\leq i\leq s}\left(\C_{\lambda_i}(-t,\bfq)+\C_{\lambda_i}(-t,\bfr)\right)\cdot 1=\G(t,\bfp,\bfq,\bfr)\prod_{1\leq i\leq s}\C_{\lambda_i}(-t,\bfp)\cdot 1.$$
We multiply by $\prod_{1\leq i\leq 1}\alpha\lambda_i$ and we extract the coefficient of $[t^{|\mu|+|\nu|}q_\mu r_\nu]$. Using \cref{eq top a coef} and \cref{lem deg bound} we obtain
\begin{align}
  (-1)^{|\mu|+|\nu|}d^\lambda_{\mu,\nu}
  &=\sum_{\kappa\atop{|\lambda|\leq|\kappa|\leq |\mu|+|\nu|}}(-1)^{|\kappa|}a^\lambda_\kappa g^\kappa_{\mu,\nu}\label{eq a structure coef}\\
  &=(-1)^{|\lambda|}g^\lambda_{\mu,\nu}+\sum_{\kappa\atop{|\lambda|<|\kappa|\leq |\mu|+|\nu|}}(-1)^{|\kappa|}a^\lambda_\kappa g^\kappa_{\mu,\nu}.\nonumber
\end{align}

Using the induction hypothesis we get 
\begin{align}
  g^\lambda_{\mu,\nu}
  &=(-1)^{|\mu|+|\nu|-|\lambda|}d^\lambda_{\mu,\nu}-\sum_{\kappa\atop{|\lambda|<|\kappa|\leq |\mu|+|\nu|}}(-1)^{|\kappa|-|\lambda|} a^\lambda_\kappa g^\kappa_{\mu,\nu}\label{eq recursive formula}\\
  &=(-1)^{|\mu|+|\nu|-|\lambda|}d^\lambda_{\mu,\nu}-\sum_{\kappa\atop{|\lambda|<|\kappa|\leq |\mu|+|\nu|}}(-1)^{|\mu|+|\nu|-|\lambda|} a^\lambda_\kappa \sum_{m\geq 0}(-1)^m\sum_{\pi_1,\dots,\pi_m\atop{|\kappa|<|\pi_1|<\dots<|\pi_m|}} a^\kappa_{\pi_1}a^{\pi_1}_{\pi_2}a^{\pi_2}_{\pi_3}\dots a^{\pi_{m-1}}_{\pi_m}d^{\pi_m}_{\mu,\nu}\nonumber\\
  &=(-1)^{|\mu|+|\nu|-|\lambda|}\sum_{m\geq 0}(-1)^m\sum_{\pi_1,\dots,\pi_m\atop{|\lambda|<|\pi_1|<\dots<|\pi_m|}} a^\lambda_{\pi_1}a^{\pi_1}_{\pi_2}a^{\pi_2}_{\pi_3}\dots a^{\pi_{m-1}}_{\pi_m}d^{\pi_m}_{\mu,\nu}\nonumber.\qedhere
\end{align}
\end{proof}
\end{thm}

We recall that we define $\G_i$ is the homogeneous part of degree $i$ in $\G$ (see also \cref{eq:def:G_k}).
We deduce from the last proof the following proposition.
\begin{prop}\label{prop characterization G}
Fix $n\geq0$. Then, \cref{eq G.1} and equations 
\begin{equation}\label{eq com C-G k}
    [t^k]\left(\C_{\ell}(-t,\bfq)+\C_{\ell}(-t,\bfr)\right)\cdot \G(t,\bfp,\bfq,\bfr)=[t^k]\G(t,\bfp,\bfq,\bfr)\cdot \C_{\ell}(-t,\bfp),
\end{equation}
for $\ell\geq 1 $ and $\ell\leq k\leq n+\ell$, characterize the operators $\G_i$ for $i\leq n$.
\end{prop}
Note that by definition, the lowest term in $\C_\ell$ as a formal power-series in $t$ has degree $\ell$. Hence, the previous equations involve only operators $\G_i$ for $i\leq n$. 

\begin{proof}
Fix $n\geq 0$.
First notice that Eqs. \eqref{eq com C-G k} imply by induction that for any partition $\lambda=[\lambda_1,\dots,\lambda_s]$ for any $|\lambda|\leq k\leq |\lambda|+n$, we have
    \begin{equation*}
    [t^k]\prod_{1\leq i\leq s}\left(\C_{\lambda_i}(-t,\bfq)+\C_{\lambda_i}(-t,\bfr)\right)\cdot \G(t,\bfp,\bfq,\bfr)
    =[t^k]\G(t,\bfp,\bfq,\bfr)\cdot \prod_{1\leq i\leq s}\C_{\lambda_i}(-t,\bfp).
\end{equation*}
 In the proof of \cref{thm g-a} these equations are used to obtain the explicit formula  \eqref{eq g-a} of $g^\lambda_{\mu,\nu}$ for $0\leq |\mu|+|\nu|-|\lambda|\leq n$. In particular, they characterize operators $\G_0$, \dots ,$\G_n$. 
\end{proof}
Actually, the operators $\C_\ell$ for $\ell\geq 1$, can be generated using only operators $\C_0$ and $\C_1$ (see \cref{thm com C}). One can use this result to show that Equations \eqref{eq com C-G k} for $\ell\in\{0,1\}$ and $1\leq k\leq n+1$ also characterize the operators $\G_i$ for $i\leq n$. In particular, each operator $\G_i$ is characterized by finitely many equations.

\subsection{\texorpdfstring{$g^\pi_{\mu,\nu}$}{g} as structure coefficients of symmetric functions}\label{ssec isomorphism}

We denote for every $\mu$
$$A_\mu^{(\alpha)}:=\sum_{\lambda}(-1)^{|\mu|}a^\lambda_{\mu} p_\lambda.$$
With this definition, multiplying \cref{eq a structure coef} by $p_\lambda$ and summing over all $\lambda$ gives
$$A_\mu^{(\alpha)}\cdot A^{(\alpha)}_\nu=\sum_{\kappa}g^\kappa_{\mu,\nu}A^{(\alpha)}_\kappa.$$
When $p_\lambda$ is thought of as a power-sum symmetric function in an underlying alphabet $\bfx$ (see \cref{ssec symmetric functions}), $A^{(\alpha)}_\mu$ becomes a symmetric function. We then have the following corollary.
\begin{cor}\label{cor iso}
The map 
$$\begin{array}{ccc}
    \mathcal S_\alpha & \longrightarrow&\mathcal S_\alpha^* \\
    A^{(\alpha)}_\mu & \longmapsto & \theta_\mu^{(\alpha)} 
\end{array}$$
is an algebra isomorphism between $\mcS$ and $\mcSstar$.
\end{cor}
\noindent For $\alpha=1$, such an isomorphism has been obtained in \cite{CorteelGoupilSchaeffer2004} using a different approach.

\subsection{A differential expression for the lower terms of \texorpdfstring{$\G$}{}}\label{ssec low terms}
In this section we prove \cref{thm low terms}. 
This proof represents no difficulty but involves some lengthy computation. 
For any $\ell,k\geq 0$, we consider the operator 
$$\C_{\ell,k}(\bfp)=(\ell+k)[t^{k+\ell}]\C_\ell(t,\bfp).$$
If $\mathcal{X}$ and $\mathcal{X'}$ are two vector spaces, we denote by $\mathcal{O}(\mathcal{X}$,$\mathcal{X'})$ the space of linear operators from $\mathcal{X}$ to $\mathcal{X'}$. We also set  $\mathcal{O}(\mathcal{X}):=\mathcal{O}(\mathcal{X},\mathcal{X})$.
Let $O$ be an operator in one alphabet such that $O(\bfp)\in\mathcal{O}\left(\mathbb Q(\alpha)[\bfp]\right)$, and let $P(\bfp,\bfq,\bfr)\in\mathcal{O}\left(\mathbb Q(\alpha)[\bfp],\mathbb Q(\alpha)[\bfq,\bfr]\right)$. We introduce their \textit{three alphabet commutator} $\left[O,P\right]_{\bfp}^{\bfq,\bfr}\in\mathcal{O}\left(\mathbb Q(\alpha)[\bfp],\mathbb Q(\alpha)[\bfq,\bfr]\right)$ defined by
$$\left[P,O\right]_{\bfp}^{\bfq,\bfr}:=P(\bfp,\bfq,\bfr)\cdot O(\bfp)-(O(\bfq)+O(\bfr))\cdot P(\bfp,\bfq,\bfr).$$

If $Q_1\in\mathcal{O}\left(\mathbb Q(\alpha)[\bfp]\right)$ and $Q_2\in\mathcal{O}\left(\mathbb Q(\alpha)[\bfq,\bfr]\right)$, it is easy to check that this commutator satisfies the relation 
\begin{align}\label{eq com 3 relation}
  \left[Q_2\cdot P\cdot Q_1,O\right]_{\bfp}^{\bfq,\bfr}
  =&[Q_2(\bfq,\bfr),O(\bfq)]\cdot P\cdot Q_1+[Q_2(\bfq,\bfr),O(\bfr)]\cdot P\cdot Q_1
\\
&+Q_2(\bfq,\bfr)\cdot \left[P,O\right]_{\bfp}^{\bfq,\bfr}\cdot Q_1(\bfp)+Q_2(\bfq,\bfr)\cdot P\cdot \left[Q_1(\bfp) ,O(\bfp)\right].\nonumber
\end{align}

We denote by $\tG_0, \tG_1,\tG_2$ the differential operators given respectively by the right-hand sides of \cref{eq G_0,eq G_1,eq G_2}. Our goal is to prove that $\G_i=\tG_i$ for $0\leq i\leq 2$. Applying \cref{prop characterization G}, it is enough to show that for any 
$\ell\geq 1$ and $0\leq i\leq 2$
\begin{equation}\label{eq com tilde G-C}
\sum_{0\leq j\leq i} \frac{(-1)^{\ell+j}}{\ell+j}\left[\tG_{i-j},\C_{\ell,j}\right]_{\bfp}^{\bfq,\bfr}=0
\end{equation}

The following lemma gives a differential  expression for the operators $\C_{\ell,j}$ for $j\leq 2$, which does not involve the alphabet $Y$. The proof is given in \cref{appendix}.
\begin{lem}\label{lem C 0-1-2}
For any $\ell\geq 1$, we have
\begin{subequations}
\begin{equation}\label{eq C_0}
  \C_{\ell,0}=p_\ell/{\alpha},  
\end{equation}
\begin{equation}\label{eq C_1}
  \C_{\ell,1}=\binom{\ell+1}{2}\frac{b}{\alpha}\cdot p_{\ell+1}+(\ell+1)\sum_{m\geq 1} p_{m+\ell+1}\frac{m\partial}{\partial p_m}+\frac{\ell+1}{2\alpha}\sum_{1\leq i\leq \ell} p_i p_{\ell+1-i},  
\end{equation}    
\begin{align}\label{eq C_2}
  \C_{\ell,2}&=
  \frac{1}{3\alpha}\binom{\ell+2}{2}\sum_{i_1+i_2+i_3=\ell+2\atop{i_1,i_2,i_3\geq 1}}p_{i_1}p_{i_2}p_{i_3}+\binom{\ell+2}{3}\frac{(3\ell+5)b^2}{4\alpha}p_{\ell+2}\\
  &+\alpha\binom{\ell+2}{2}\sum_{k,m\geq 1}p_{\ell+k+m+2}\frac{m\partial}{\partial p_m}\frac{k\partial}{\partial p_k}+\binom{\ell+2}{2}\sum_{m\geq1}b(\ell+m+1)p_{m+\ell+2}\frac{m\partial}{\partial p_m}\nonumber\\
  &+\sum_{i_1+i_2=\ell+2\atop{i_1,i_2\geq 1}}\frac{b\cdot (\ell+2)\left((\ell+1)^2-i_1i_2\right)}{2\alpha}p_{i_1}p_{i_2}+\binom{\ell+3}{4}p_{\ell+2}\nonumber \\
  &+\binom{\ell+2}{2}\sum_{m\geq 1}\sum_{i_1+i_2=\ell+m+2\atop{i_1,i_2\geq1}}p_{i_1}p_{i_2}\frac{m\partial}{\partial p_m}.\nonumber
\end{align}

\end{subequations}
\end{lem}

In the following lemma we establish some useful commutation relations for the operator~$\Psi$.
\begin{lem}\label{lem Psi commmutator}
    We have the following relations between operators in $\mathcal{O}(\mathbb Q(\alpha)[\bfp],\mathbb Q(\alpha)[\bfq,\bfr])$ 
    $$\Psi\cdot p_\ell=(q_\ell+r_\ell)\cdot \Psi,\quad \textit{i.e.}\quad  \left[\Psi,p_\ell\right]^{\bfq,\bfr}_\bfp=0$$
    $$\Psi\cdot \frac{\partial}{\partial p_\ell}=\frac{\partial}{\partial q_\ell}\cdot \Psi=\frac{\partial}{\partial r_\ell}\cdot \Psi.$$
    Moreover, 
    $$\left[\Psi,p_i\frac{\partial}{\partial p_j}\right]^{\bfq,\bfr}_\bfp=\left[\Psi,p_i\frac{\partial}{\partial p_{j_1}}\frac{\partial}{\partial p_{j_2}}\right]^{\bfq,\bfr}_\bfp=0,$$
    $$\left[\Psi,p_{i_1}p_{i_2}\right]^{\bfq,\bfr}_\bfp=(q_{i_1}r_{i_2}+r_{i_1}q_{i_2})\Psi,$$
    $$\left[\Psi,p_{i_1}p_{i_2}\frac{\partial}{\partial p_j}\right]^{\bfq,\bfr}_\bfp=q_{i_1}r_{i_2}\Psi \frac{\partial}{\partial p_j}+r_{i_1}q_{i_2}\Psi \frac{\partial}{\partial p_j}.$$
    \begin{proof}
        The first equation is immediate from the definition of the operator $\Psi$. Let us show the second equation.
        Fix a partition $\lambda$. If $m_\ell(\lambda)=0$ then 
        $$\Psi\cdot \frac{\partial}{\partial p_\ell}p_\lambda=\frac{\partial}{\partial q_\ell}\cdot \Psi p_\lambda=0.$$
        Otherwise, we denote by $\mu$ the partition obtained from $\lambda$ by erasing a part of size $\ell$.  Then 
        $$\Psi\cdot \frac{\partial}{\partial p_\ell}p_\lambda=m_\ell(\lambda)\Psi p_\mu= m_\ell(\lambda)\prod_{i\in\mu}(q_i+r_i).$$
        On the other hand,
        $$\frac{\partial}{\partial q_\ell}\cdot \Psi p_\lambda=\frac{\partial}{\partial q_\ell}\prod_{i\in\lambda}(q_i+r_i)=m_\ell(\lambda)\prod_{i\in\mu}(q_i+r_i).$$
       The last three equations follow from the first ones.  
    \end{proof}
\end{lem}

We now prove \cref{thm low terms}.
\begin{proof}[Proof of \cref{thm low terms}]
    From \cref{lem Psi commmutator} and \cref{eq C_0,eq C_1}, we have
    $$\frac{1}{\ell}\left[\tG_0,\C_{\ell,0}\right]_{\bfp}^{\bfq,\bfr}=0,$$
    $$\frac{1}{\ell+1}\left[\tG_0,\C_{\ell,1}\right]_{\bfp}^{\bfq,\bfr}=\frac{1}{\alpha}\sum_{m_1+m_2=\ell+1\atop{m_1,m_2\geq 1}}q_{m_1}r_{m_2}\Psi=\frac{1}{\ell}\left[\tG_1,\C_{\ell,0}\right]_{\bfp}^{\bfq,\bfr}.$$
These two equations together with \cref{prop characterization G} give \cref{eq G_0,eq G_1}. Let us now prove \cref{eq G_2}. Using \cref{lem Psi commmutator} and \cref{eq C_2}, we get
    \begin{align*}
      \frac{1}{\ell+2}\left[\tG_0,\C_{\ell,2}\right]_{\bfp}^{\bfq,\bfr}=
      &\frac{\ell+1}{2\alpha}\sum_{i_1+i_2+i_3=\ell+2\atop{i_1,i_2,i_3\geq1}}(q_{i_1}q_{i_2}r_{i_3}+q_{i_1}r_{i_2}r_{i_3})\cdot \Psi\\
      &+\sum_{i_1+i_2=\ell+2\atop{i_1,i_3\geq 1}}\frac{b\cdot \left((\ell+1)^2-i_1i_2\right)}{2\alpha}q_{i_1}r_{i_2}\cdot \Psi\\
      &+(\ell+1)\sum_{m\geq 1}\sum_{i_1+i_2=\ell+m+2\atop{i_1,i_2\geq1}}q_{i_1}r_{i_2}\cdot \Psi\cdot\frac{m\partial}{\partial p_m}.
    \end{align*}
Moreover, 
    \begin{align*}
      \frac{1}{\ell}\left[\tG_2,\C_{\ell,0}\right]_{\bfp}^{\bfq,\bfr}  
      =&\frac{1}{2\alpha}\sum_{m_1+m_2=\ell+2\atop{m_1,m_2\geq 1}}b(m_1-1)(m_2-1)q_{m_1}r_{m_2}\Psi\label{eq G_2}\\ 
  &\nonumber+\frac{1}{2\alpha}\sum_{m_1+m_2+m_3=\ell+2\atop{m_1,m_2,m_3\geq 1}}(m_1-1)(q_{m_1}r_{m_2}r_{m_3}+r_{m_1}q_{m_2}q_{m_3})\Psi\\\nonumber
  &+\sum_{m\geq1}\sum_{i_1+i_2=m+\ell+2\atop{i_1,i_2\geq1}}\min(\ell,m,i_1-1,i_2-1)q_{i_1}r_{i_2}\Psi\frac{m\partial}{\partial p_m}\\\nonumber
&+\frac{1}{\alpha}\sum_{m\geq1}\sum_{k_1+k_2=\ell+1\atop{k_1,k_2\geq1}}\sum_{m_1+m_2=m+1\atop{m_1,m_2\geq1}}q_{m_1}q_{k_1}r_{m_2}r_{k_2}\Psi\frac{m\partial}{\partial p_m}.\nonumber
    \end{align*}
Applying \cref{eq com 3 relation}, we get 
\begin{align*}
  \left[\tG_1,\C_{\ell,1}\right]_{\bfp}^{\bfq,\bfr}
  =&\sum_{m\geq 1}\sum_{\genfrac..{0pt}{2}{m_1+m_2=m+1}{m_1,m_2\geq 1}}\left(
  [q_{m_1}r_{m_2},\C_{\ell,1}(\bfq)]\cdot\Psi\cdot\frac{m\partial}{\partial p_m}+[q_{m_1}r_{m_2},\C_{\ell,1}(\bfr)]\cdot\Psi\cdot\frac{m\partial}{\partial p_m}\right.\\
  &\left.+
  q_{m_1}r_{m_2}\cdot \left[\Psi,\C_{\ell,1}\right]_{\bfp}^{\bfq,\bfr}\cdot\frac{m\partial}{\partial p_m}
  +q_{m_1}r_{m_2}\cdot \Psi\cdot \left[\frac{m\partial}{\partial p_m},\C_{\ell,1}(\bfp)\right] \right).
  \end{align*}
  Hence,
  \begin{align*}
  \frac{1}{\ell+1}\left[\tG_1,\C_{\ell,1}\right]_{\bfp}^{\bfq,\bfr}
  =&\sum_{m\geq 1}\sum_{\genfrac..{0pt}{2}{m_1+m_2=m+1}{m_1,m_2\geq 1}}
  \left(-\left( m_1q_{m_1+\ell+1}r_{m_2}+m_2 q_{m_1}r_{m_2+\ell+1}\right)\cdot \Psi\cdot \frac{m\partial}{\partial p_m}\right.\\
  &\left.+\frac{1}{\alpha}q_{m_1}r_{m_2}\sum_{k_1+k_2=\ell+1\atop{k_1,k_2\geq 1}}q_{k_1}r_{k_2}\Psi\frac{m\partial}{\partial p_m}\right)+\frac{b(\ell+1)\ell}{2\alpha}\sum_{\genfrac..{0pt}{2}{m_1+m_2=\ell+2}{m_1,m_2\geq 1}}q_{m_1}r_{m_2}\cdot \Psi\\
  &+\sum_{m\geq 1}\sum_{\genfrac..{0pt}{2}{m_1+m_2=m+\ell+2}{m_1,m_2\geq 1}}(m+\ell+1)q_{m_1}r_{m_2}\cdot\Psi\cdot\frac{m\partial}{\partial p_m}\\
  &+\frac{1}{\alpha}\sum_{\genfrac..{0pt}{2}{m_1+m_2+m_3=\ell+2}{m_1,m_2,m_3\geq 1}}(m_1+m_2-1) (q_{m_1}r_{m_2}q_{m_3}+q_{m_1}r_{m_2}r_{m_3})\cdot \Psi.
\end{align*}
The last sum can be symmetrized as follows
$$\frac{1}{2\alpha}\sum_{\genfrac..{0pt}{2}{m_1+m_2+m_3=\ell+2}{m_1,m_2,m_3\geq 1}}(2m_1+m_2+m_3-2) (r_{m_1}q_{m_2}q_{m_3}+q_{m_1}r_{m_2}r_{m_3})\cdot \Psi.$$
One may also notice that for any $m,\ell,i_1,i_2\geq 1$, such that $i_1+i_2=m+\ell+2$ we have
$$m-\mathbbm 1_{i_1\geq \ell+2}(i_1-\ell-1)-\mathbbm 1_{i_2\geq \ell+2}(i_2-\ell-1)=\min(m,\ell,i_1-1,i_2-1).$$
Using these two remarks and combining the three equations above, we get
$$\frac{1}{\ell+2}\left[\tG_0,\C_{\ell,2}\right]_{\bfp}^{\bfq,\bfr}-\frac{1}{\ell+1}\left[\tG_1,\C_{\ell,1}\right]_{\bfp}^{\bfq,\bfr}+\frac{1}{\ell}\left[\tG_2,\C_{\ell,0}\right]_{\bfp}^{\bfq,\bfr}  =0.$$
which gives \cref{eq com tilde G-C} for $n=2$ and finishes the proof of the theorem.
\end{proof}




\section{Equations for connected series}\label{sec connected}
In this section we consider a connected version $\hG^{(\alpha)}$ of the series $G^{(\alpha)}$. We establish some general properties about the series $\hG^{(\alpha)}$ and we derive from the main theorem a family of differential equation for this series.
\subsection{Connected series}
We introduce the two series
$$\hG\equiv\hG^{(\alpha)}(t,\bfp,\bfq,\bfr):=\alpha\cdot\log(G^{(\alpha)}(t,\bfp,\bfq,\bfr)),$$
and
$$ \htau \equiv\htau^{(\alpha)}(t,\bfp,\bfq,\bfr):=\alpha\cdot\log(\tau^{(\alpha)}(t,\bfp,\bfq,\bfr)),$$
where $\tau^{\alpha}$ is the series defined in \cref{eq:tau}.
The series $\hG$ and $\htau$ are well defined in $\mathbb Q(\alpha)[t,\bfp]\llbracket \bfq,\bfr\rrbracket\cap\mathbb Q(\alpha)[t,\bfq,\bfr]\llbracket \bfp\rrbracket$, since 
$$[p_\emptyset]G^{(\alpha)}=[p_\emptyset] \htau^{(\alpha)} =[q_\emptyset r_\emptyset ]G^{(\alpha)}=[q_\emptyset r_\emptyset ]\tau^{(\alpha)}=1,$$
where $\emptyset$ denotes the empty integer partition, see \cref{lem deg bound}. 

By \cref{prop H-G} and \cref{thm coef c b=0}, the series $\tau^{(\alpha)}$ (resp. $\hG^{(\alpha)}$)  is a generating series of connected hypermaps (resp. connected hypermaps with marked faces) when $\alpha\in\{1,2\}$. These two series are related by a variant of \cref{eq G-tau}.
\begin{lem}\label{lem hG-htau}
We have,
$$\hG(t,\bfp,\bfq,\bfr)=-p_1/t+\exp\left(\frac{\partial}{t\partial q_1}\right)\exp\left(\frac{\partial}{t\partial r_1}\right)\htau(t,\bfp,\bfq,\bfr).$$
\begin{proof}
First, notice that the operator $\exp\left(\frac{\partial}{t\partial q_1}\right)$  is well defined on $\mathbb  Q(\alpha)[\bfq,\bfr,t,1/t]\llbracket \bfp\rrbracket$, and for any $A,B\in\mathbb Q(\alpha)[\bfq,\bfr,t,1/t]\llbracket \bfp\rrbracket,$
    $$\exp\left(\frac{\partial}{t\partial q_1}\right)\cdot (AB)=\left(\exp\left(\frac{\partial}{t \partial q_1}\right)\cdot A\right)
    \left(\exp\left(\frac{\partial}{t\partial q_1}\right)\cdot B\right).$$
    Since $\htau(t,\bfp,\bfq,\bfr)\in \mathbb Q(\alpha)[\bfq,\bfr,t]\llbracket \bfp\rrbracket\subset \mathbb Q(\alpha)[\bfq,\bfr,t,1/t]\llbracket \bfp\rrbracket$, we get
    \begin{align*}
      \exp\left(\frac{\partial}{t\partial q_1}\right)\exp\left(\frac{\partial}{t\partial r_1}\right)\cdot \tau(t,\bfp,\bfq,\bfr)
      &=\exp\left(\frac{\partial}{t\partial q_1}\right)\exp\left(\frac{\partial}{t\partial r_1}\right)\cdot \sum_{k\geq 0}\frac{\htau(t,\bfp,\bfq,\bfr)^k}{\alpha^k k!}\\
      &=\sum_{k\geq 0}\frac{1}{k!} \left(\exp\left( \frac{\partial}{t\partial q_1}\right) \exp\left(\frac{\partial}{t\partial r_1}\right)\cdot \frac{\htau(t,\bfp,\bfq,\bfr)}{\alpha}\right)^k\\
      &=\exp\Bigg(\exp\left(\frac{\partial}{t\partial q_1}\right)\exp\left(\frac{\partial}{t\partial r_1}\right)\cdot \frac{\htau(t,\bfp,\bfq,\bfr)}{\alpha}\Bigg)
      .  
    \end{align*}
    
    We conclude by substituting this formula in \cref{eq G-tau}.
\end{proof}    
\end{lem}


The coefficients $h^\pi_{\mu,\nu}$, introduced in \cite{GouldenJackson1996} are defined by 
$$\htau(t,\bfp,\bfq,\bfr)=\sum_{n\geq 1}t^n\sum_{\pi,\mu,\nu\vdash n}\frac{h^\pi_{\mu,\nu}(\alpha)}{n}p_\pi q_\mu r_\nu.$$
These coefficients are the object of the hypermaps-Jack conjecture (known also as the $b$-conjecture), see  \cite[Conjecture 6.3]{GouldenJackson1996}. Similarly, we consider the coefficients $\hat{g}^\pi_{\mu,\nu}(\alpha)$
$$\hG(t,\bfp,\bfq,\bfr) =\sum_{\pi,\mu,\nu}\frac{\hat{g}^\pi_{\mu,\nu}(\alpha)}{|\pi|}t^{|\mu|+|\nu|-|\pi|}p_\pi q_\mu r_\nu.$$
The following polynomiality result is due to Féray and Do\l{}e\k{}ga.
\begin{thm}[\cite{DolegaFeray2017}]\label{thm h coef}
    For any partitions $\pi,\mu,\nu\vdash n\geq 1$, the coefficient $h^\pi_{\mu,\nu}$ is polynomial in $b$ and 
    $$\deg(h^\pi_{\mu,\nu})\leq n+2-\left(\ell(\pi)+\ell(\mu)+\ell(\nu)\right).$$
\end{thm}

\noindent We deduce the following corollary.
\begin{cor}
    For any partitions $\pi,\mu$ and $\nu$, the coefficients $\hat{g}^\pi_{\mu,\nu}$ is polynomial in $b$, and 
    $$\deg(\hat{g}^\pi_{\mu,\nu})\leq 2+|\mu|-\ell(\mu)+|\nu|-\ell(\nu)-(|\pi|+\ell(\pi)).$$
    \begin{proof}
        From \cref{lem hG-htau}, we get that for any $|\pi|\geq \max(|\mu|,|\nu|)$, the coefficient $\hat{g}^\pi_{\mu,\nu}$ is given by 
        $$\hat{g}^{\pi}_{\mu,\nu} =\left\{
        \begin{array}{ll}
            0 & \text{ if } (\pi,\mu,\nu) =([1],[0],[0])  \\
            \binom{m_1(\mu)+|\pi|-|\mu|}{m_1(\mu)}\binom{m_1(\nu)+|\pi|-|\nu|}{m_1(\nu)}h^\pi_{\mu\cup 1^{|\pi|-|\mu|},\nu\cup1^{|\pi|-|\nu|}}, 
            &\text{ otherwise.} 
        \end{array}\right.
        $$
From \cref{thm h coef}, we get that $\hat{g}^\pi_{\mu,\nu}$ is polynomial and 
\begin{align*}
\deg\left(\hat{g}^\pi_{\mu,\nu}\right)
&=\deg(h^\pi_{\mu\cup 1^{|\pi|-|\mu|},\nu\cup1^{|\pi|-|\pi|}})\\
&\leq |\pi|+2-\Big(\ell(\pi)+\ell\left(\mu\cup 1^{|\pi|-|\mu|}\right)+\ell\left(\nu\cup1^{|\pi|-|\nu|}\right)\Big)\\
&=2+|\mu|-\ell(\mu)+|\nu|-\ell(\nu)-(|\pi|+\ell(\pi)).\qedhere
\end{align*}
    \end{proof}
\end{cor}

    
\subsection{Dual operators}\label{ssec dual operators}
The purpose of this subsection is to give a differential expression of dual operators $\B_n$.
For this, we introduce the scalar product $\langle,\rangle_Y$ on
$\tildePY$ (see \cref{ssec op Bn}), defined by
$$\left\{
\begin{array}{l}
     \langle p_\lambda, p_\mu\rangle_Y=\delta_{\lambda,\mu}\alpha^{\ell(\lambda)}z_\lambda=\langle p_\lambda, p_\mu\rangle;  \\
     \langle p_\lambda, y_i p_\mu\rangle_Y=0;\\
     \langle y_i p_\lambda, y_j p_\mu\rangle_Y=\delta_{i,j}\delta_{\lambda,\mu}\alpha^{\ell(\lambda)+1}z_\lambda,
\end{array}
\right.$$
for any  $\lambda,\mu\in \mathbb Y$ and $i,j\geq 0$.

If $A_Y$ is an operator on $\PY$, then we denote by $A_Y^\perp$ its dual operator with respect to $\langle,\rangle_Y$.
We deduce from the definitions the following differential expressions for the catalytic operators; $$(y_i)^\perp=\frac{\alpha \partial}{\partial y_i}, \text{ for any }i\geq 0.$$
\begin{equation*}
    Y_+^\perp=\sum_{i\geq2} y_{i-1}\frac{\partial}{\partial y_i},
\end{equation*}

\begin{equation*}
    \Theta_Y^\perp(\bfp)=\sum_{i\geq 1}y_i\frac{i\partial}{\partial p_i},
\end{equation*}

\begin{equation*}
    \GY^\perp(\bfp)=\sum_{i,j\geq1}y_{i-1} p_{j}\frac{\partial}{\partial y_{i+j}}
    +(1+b)\cdot\sum_{i,j\geq1} y_{i+j-1}\frac{j\partial^2}{\partial y_i\partial p_{j}} 
    +b\cdot \sum_{i\geq2}y_{i-1}\frac{(i-1)\partial}{\partial y_i},
\end{equation*}

\begin{equation}\label{eq B_n dual}
    \mathcal{B}_n^\perp(\bfp,u)=\frac{\partial}{\partial y_0}\left(\GY^\perp+uY_+^\perp\right)^{n}\Theta_Y^\perp.
\end{equation}
We recall that $b:=\alpha-1$.

\subsection{Differential equation for the series of connected maps}
We denote for each $m\geq 1$
$$\hG^{[m]}_{\bfp}=\frac{m\partial}{\partial p_m}\hG,\qquad \hG^{[m]}_{\bfq}=\frac{m\partial}{\partial q_m}\hG,\qquad \text{}\hG^{[m]}_{\bfr}=\frac{m\partial}{\partial r_m}\hG.$$

\begin{prop}\label{prop diff expression}
Fix $n\geq 1$.
Then, we have the equality between operators in $\mathcal{O}(\mcA)$
\begin{equation*}
  \B_n(\bfq,u)\cdot G^{(\alpha)}
  =G^{(\alpha)}.\Theta_Y(\bfq)\left(\GY(\bfq)+u Y_++\sum_{i,j\geq 1}y_{i+j}\frac{\partial}{\partial y_{i-1}}\hG^{[j]}_{\bfq}\right)^{n}\frac{y_0}{1+b}.
\end{equation*}
Here, $G^{(\alpha)}$ acts on $\mcA$ by multiplication. Similarly,
\begin{multline*}
  {\B_n}^\perp(\bfp,u)\cdot G^{(\alpha)}
  \\=G^{(\alpha)}\cdot \frac{\partial}{\partial y_0}\left(\GY^\perp(\bfp)+u Y^\perp_++\sum_{i,j\geq 1}y_{i+j-1}\frac{\partial}{\partial y_{i}}\hG^{[j]}_{\bfp}\right)^{n} \left(\Theta^\perp_Y(\bfp)+\sum_{i\geq 1}y_i\hG^{[i]}_{\bfp}\right).
\end{multline*}

\begin{proof}
This is a consequence of the catalytic differential expressions of operators $\B_n$ and ${\B_n}^\perp$ given resp. in \cref{eq def Bn} and  \cref{eq B_n dual}.  We also use the fact that
\begin{equation*}
    \left[(1+b)\frac{j\partial}{\partial p_j},G^{(\alpha)}\right]=(1+b)\frac{j\partial G^{(\alpha)}}{\partial p_j}=G^{(\alpha)}.\hG^{[j]}_{\bfp}\quad
    \text{and}\quad \left[\frac{\partial}{\partial y_i},G^{(\alpha)}\right]=0.\qedhere
\end{equation*}
\end{proof}
\end{prop}

We deduce from \cref{thm diff eq} and \cref{prop diff expression} the following theorem.
\begin{thm}\label{thm diff for connected}
The series $\hG$ satisfies the following differential equation:
    \begin{multline*}
        \sum_{n\geq 1}\frac{(-1)^n}{n}\frac{\partial}{\partial y_0}\left(\GY^\perp(\bfp)+u Y^\perp_++\sum_{i,j\geq 1}y_{i+j-1}\frac{\partial}{\partial y_{i}}\hG^{[j]}_{\bfp}\right)^{n} \left(\sum_{i\geq 1}y_i\hG^{[i]}_{\bfp}\right)\cdot 1\\
        =\sum_{n\geq 1}\frac{(-1)^n}{n}\Theta_Y(\bfq)\left(\GY(\bfq)+u Y_++\sum_{i,j\geq 1}y_{i+j}\frac{\partial}{\partial y_{i-1}}\hG^{[j]}_{\bfq}\right)^{n}\frac{y_0}{1+b}\cdot 1\\
        +\sum_{n\geq 1}\frac{(-1)^n}{n}\Theta_Y(\bfr)\left(\GY(\bfr)+u Y_++\sum_{i,j\geq 1}y_{i+j}\frac{\partial}{\partial y_{i-1}}\hG^{[j]}_{\bfr}\right)^{n}\frac{y_0}{1+b}\cdot 1.
    \end{multline*}
\end{thm}


\appendix

\section{Proof of Lemma \ref{lem C 0-1-2}}\label{appendix}
We prove in this appendix \cref{lem C 0-1-2}. In order to obtain an explicit differential expression for operators $\C_0$, $\C_1$ and $\C_2$ we develop the catalytic expressions given in \cref{ssec op Bn} for these operators. Since the computations are lengthy for the operator $\C_2$, we explain the important steps of the proof without giving all the details.  

It is direct from the definitions that for any $i\geq 0$ we have
$$Y_+^i\frac{y_0}{\alpha}=\frac{y_i}{\alpha},\quad \text{as operators on $\PY$.}$$
We apply $\Theta_Y$ to obtain \cref{eq C_0}. By applying $\Gamma_Y$ on the last equation we get;
    $$\Gamma_Y Y_+^i\frac{y_0}{\alpha}=\frac{ib}{\alpha} y_{i+1}+ \sum_{m\geq1}y_{i+m+1}\frac{m\partial}{\partial p_m}+\frac{1}{\alpha}\sum_{1\leq j\leq i}y_{i-j+1}p_{j},$$
 Hence, for any $i_1,i_2\geq 0$, we have
 \begin{align}\label{eq proof C1}
  Y_+^{i_2} \Gamma_Y Y_+^{i_1}\frac{y_0}{\alpha}=\frac{i_1 b}{\alpha} y_{i_1+i_2+1}
+\sum_{m\geq1}y_{i_1+i_2+m+1}\frac{m\partial}{\partial p_m}
+\frac{1}{\alpha}\sum_{1\leq j\leq i_1}y_{i_1+i_2-j+1}p_{j}.   
 \end{align}

\noindent We apply $\Theta_Y$ and we sum over all tuples $(i_1,i_2)$ such that $i_1+i_2=\ell$ to obtain
\begin{align*}
  \C_{\ell,1}
  &=\frac{b}{\alpha}p_{\ell+1} \sum_{i_1+i_2=\ell\atop{i_1,i_2\geq 0}}i_1+(\ell+1)\sum_{m\geq1}p_{\ell+m+1}\frac{m\partial}{\partial p_m}+\frac{1}{\alpha}\sum_{j_2+j_2=\ell+1\atop{j_1,j_2\geq 1}}p_{j_1}p_{j_2}\sum_{i_1+i_2=\ell\atop{i_1,i_2\geq 0}}\mathbbm 1_{i_1\geq j_1}\\
  &=\frac{b}{\alpha}\binom{\ell+2}{2}p_{\ell+1} +(\ell+1)\sum_{m\geq1}p_{\ell+m+1}\frac{m\partial}{\partial p_m}+\frac{1}{\alpha}\sum_{j_2+j_2=\ell+1\atop{j_1,j_2\geq 1}}p_{j_1}p_{j_2}j_2.\\
\end{align*}
In order to obtain \cref{eq C_1}, we \textit{symmetrize} the last sum with respect to $(j_1,j_2)$;
\begin{align*}
  \sum_{j_2+j_2=\ell+1\atop{j_1,j_2\geq 1}}p_{j_1}p_{j_2}j_2
  &=\frac{1}{2}\left(\sum_{j_2+j_2=\ell+1\atop{j_1,j_2\geq 1}}p_{j_1}p_{j_2}j_2+\sum_{j_2+j_2=\ell+1\atop{j_1,j_2\geq 1}}p_{j_1}p_{j_2}j_1\right) \\
  &=\frac{1}{2}\sum_{j_2+j_2=\ell+1\atop{j_1,j_2\geq 1}}p_{j_1}p_{j_2}(\ell+1).
\end{align*}
This idea will be used repeatedly in the proof of \cref{eq C_2} which we now explain. We start by applying $\Gamma_Y$ on \cref{eq proof C1};

\begin{align*}
    \Gamma_Y Y_+^{i_2} \Gamma_Y Y_+^{i_1}\frac{y_0}{\alpha}
    =&\frac{b^2}{\alpha} i_1(i_1+i_2+1) y_{i_1+i_2+2}+i_1b \sum_{k\geq 1}y_{i_1+i_2+k+2}\frac{k\partial}{\partial p_k}\\
    &+\frac{i_1b}{\alpha}\sum_{j=1}^{i_1+i_2+1}y_{i_1+i_2-j+2}p_{j}+b(i_1+i_2+m+1)\sum_{m\geq1}y_{i_1+i_2+m+2}\frac{m\partial}{\partial p_m}\\
    &+\alpha\sum_{m,k\geq 1}y_{i_1+i_2+m+k+2}\frac{m\partial}{\partial p_m}\frac{k\partial}{\partial p_k}+\sum_{m\geq1}\sum_{j=1}^{i_1+ i_2+m+1}y_{i_1+i_2+m+2-j}p_{j}\frac{m\partial}{\partial p_m}\\
    &+\frac{b}{\alpha}(i_1+i_2-j+1)\sum_{1\leq j\leq i_1}y_{i_1+i_2-j+2}p_{j}+\sum_{1\leq j\leq i_1}\sum_{k\geq 1}y_{i_1+i_2-j+k+1}\frac{k\partial}{\partial p_k}p_{j}\\
    &+\sum_{j=1}^{i_1}\sum_{j'=1}^{i_1+i_2-j+1}y_{i_1+i_2-j+1} p_{j'}p_{j}.
\end{align*}
Fix three integers  $i_1,i_2,i_3\geq 0$ such that $i_1+i_2+i_3=\ell$.
We apply $Y_+^{i_3}$ on the last equation and we regroup the terms of the same type. We get
\begin{align*}
    Y_+^{i_3}\Gamma_Y Y_+^{i_2} \Gamma_Y Y_+^{i_1}\frac{y_0}{\alpha}
    =&\frac{b^2}{\alpha} i_1(i_1+i_2+1) y_{\ell+2}+(2i_1+i_2+m+1)b \sum_{m\geq 1}y_{\ell+m+2}\frac{m\partial}{\partial p_m}\\
    &+\frac{b}{\alpha}\sum_{j=1}^{\ell+1}(i_1\mathbbm 1 _{j\leq i_1+i_2+1}+(i_1+i_2-j+1)\mathbbm 1_{j\leq i_1})y_{\ell-j+2}p_{j}\\
    &+\alpha\sum_{m,k\geq 1}y_{\ell+m+k+2}\frac{m\partial}{\partial p_m}\frac{k\partial}{\partial p_k}
    +\sum_{m\geq1}\sum_{j=1}^{\ell+m+1}\mathbbm 1_{j\leq i_1+i_2+m+1}y_{\ell+m+2-j}p_{j}\frac{m\partial}{\partial p_m}\\
    &+\sum_{1\leq j\leq \ell}\sum_{k\geq 1}y_{\ell-j+k+2}\mathbbm 1_{j\leq i_1}(p_{j}\frac{k\partial}{\partial p_k}+k\delta_{k,j})\\
    &+\sum_{j,j'\geq1\atop{j+j'\leq \ell+1}}\mathbbm 1_{j\leq i_1}\mathbbm 1_{j'\leq i_1+i_2-j+1}y_{\ell-j+2} p_{j'}p_{j}.
\end{align*}
By applying $\Theta_Y$ and taking the sum over all tuples $(i_1,i_2,i_3)$, we get\footnote{This step involves the computations of sums of some polynomial expression in the variables $i_1, i_2$ and $i_3$ which can be easily checked using a software of formal computation (Maple for example).} 
 
\begin{align*}
  \C_{\ell,2}
  &=\binom{\ell+2}{3}\frac{(3\ell+5)b^2}{4\alpha}p_{\ell+2}+b\binom{\ell+2}{2}\sum_{m\geq 1}p_{m+\ell+2}(m+\ell+1)\frac{m\partial}{\partial p_m}\\
  &+\frac{b}{2\alpha}\sum_{j_1+j_2=\ell+2\atop{j_1,j_2\geq 1}}j_2((\ell+1)^2-j_1 j_2)p_{j_1}p_{j_2}
  +\alpha\binom{\ell+2}{2}\sum_{m,k\geq 1}y_{\ell+m+k+2}\frac{m\partial}{\partial p_m}\frac{k\partial}{\partial p_k}+\binom{\ell+3}{4}p_\ell\\
  &+\binom{\ell+2}{2}\sum_{m\geq1 }\sum_{j_1+j_2=\ell+m+2\atop{j_1,j_2\geq 1}} p_{j_1}p_{j_2}\frac{m\partial}{\partial p_m}
  +\sum_{j_1,j_2,j_3\geq 1\atop{j_1+j_2+j_3=\ell+2}}\frac{j_3(2j_2+j_3-1)}{2}p_{j_1}p_{j_2}p_{j_3}.
\end{align*}
In order to conclude, we first symmetrize the third sum; we use the fact that $(\ell+1)^2-j_1 j_2)p_{j_1}p_{j_2}$ is symmetric in $j_1$ and $j_2$. Moreover, we symmetrize the last sum in two steps as follows; $j_3p_{j_1}p_{j_2}p_{j_3}$ is symmetric in $j_1$ and $j_2$, hence
\begin{align*}
  \frac{1}{2}\sum_{j_1,j_2,j_3\geq 1\atop{j_1+j_2+j_3=\ell+2}}j_3(2j_2+j_3-1)p_{j_1}p_{j_2}p_{j_3}
&=\frac{1}{4}\sum_{j_1,j_2,j_3\geq 1\atop{j_1+j_2+j_3=\ell+2}}j_3p_{j_1}p_{j_2}p_{j_3}(2j_2+j_3-1+2j_1+j_3-1)  \\
&=\frac{1}{2}\sum_{j_1,j_2,j_3\geq 1\atop{j_1+j_2+j_3=\ell+2}}j_3p_{j_1}p_{j_2}p_{j_3}(\ell+1).
\end{align*}
Finally, notice that $p_{j_1}p_{j_2}p_{j_3}$ is symmetric in $j_1$, $j_2$ and $j_3$. We then take the average between the three indexing tuples $(j_1,j_2,j_3)$, $(j_2,j_3,j_1)$ and $(j_3,j_1,j_2)$. This finishes the proof of \cref{eq C_2} and hence the proof of the lemma.

\section*{Acknowledgment} The author is very grateful to his advisors Valentin Féray and Guillaume Chapuy for many interesting discussions and for reviewing this paper several times. He also thanks the anonymous referees whose comments made this paper clearer.

\bibliographystyle{amsalpha}
\bibliography{biblio}
\end{document}